\documentclass[11pt, a4paper, reqno]{amsart}
\usepackage{amsmath,amsfonts, amsthm,amssymb,amsfonts,amscd, graphicx, hyperref, enumerate,mathtools, mathrsfs, dsfont, upgreek, appendix}

\newtheorem{theorem}{Theorem}[section]
\newtheorem{corollary}[theorem]{Corollary}
\newtheorem{lemma}[theorem]{Lemma}

\newtheorem{prop}[theorem]{Proposition}
\theoremstyle{remark}
\newtheorem{remark}[theorem]{Remark}
\usepackage[left=2.5cm, right=2.5 cm]{geometry}

\newcommand{\eps}{\varepsilon}
\makeatletter

\hypersetup{
  colorlinks   = true, 
  urlcolor     = black, 
  linkcolor    = black, 
  citecolor    = black 
}

\mathtoolsset{showonlyrefs=true}

\newcommand{\norm}[1]{\lVert#1\rVert} 
\DeclareMathOperator*{\dist}{dist} 
\DeclareMathOperator*{\supp}{supp} 

\newcommand*{\R}{\ensuremath{\mathbb{R}}}

\AtBeginDocument{
\let\div\undefined\DeclareMathOperator{\div}{div} 
}
\AtBeginDocument{
}
\makeindex

\author[Colombo]{Maria Colombo}
\address{EPFL SB, Station 8, 
CH-1015 Lausanne, Switzerland
}\address{Institute for advanced study, 1 Einstein Dr, Princeton, NJ 08540
}
\email{maria.colombo@epfl.ch}

\author[Haffter]{Silja Haffter
}
\address{EPFL SB, Station 8, 
CH-1015 Lausanne, Switzerland
}
\email{silja.haffter@epfl.ch}

\begin{document}
\title[Global regularity for the wave equation with slightly supercritical power]{Global regularity for the nonlinear wave equation with slightly supercritical power
}
\maketitle
\begin{abstract} We consider the defocusing nonlinear wave equation $\Box u = \lvert u \rvert^{p-1} u$ in $\R^3 \times[0,\infty)$. We prove that  for any initial datum with a scaling-subcritical norm bounded by $M_0$ the equation is globally well-posed for $p=5+\delta$ where $\delta \in (0,\delta_0(M_0))$. 
\end{abstract}

\section{Introduction}
We consider the Cauchy problem for the nonlinear defocusing wave equation on $\mathbb{R}^3$, that is
\begin{align} \label{eq:nlw}
\begin{cases} \Box u &= \lvert u \rvert^{p-1} u \\
(u, \partial_t u)(\cdot, 0)&= (u_0, u_1) \in (\dot{H}^1 \cap \dot{H}^2) \times H^1 \,  ,
\end{cases}
\end{align}
where $u: \mathbb{R}^3 \times I \to \mathbb{R}$, $p>1$ and $\Box=-\partial_{tt} + \Delta$ is the D'Alembertian. For sufficiently regular solutions of \eqref{eq:nlw} the energy
\begin{equation}
E(u)(t):= \int \frac12 \lvert \partial_t u\rvert^2 + \frac 12 \lvert \nabla u \rvert^2 + \frac{\lvert u \rvert^{p+1}}{p+1} \, \mathrm{d}x
\end{equation}
is conserved, i.e. $E(t)=E$. Moreover, there is a natural scaling associated to \eqref{eq:nlw}: For $\lambda>0$ the map
\begin{equation}\label{eq:scaling}
 u \mapsto u_\lambda(t,x)= \lambda^{\frac{2}{p-1}} u(\lambda x, \lambda t)
\end{equation}
preserves solutions of \eqref{eq:nlw}. Correspondingly, the energy rescales like $E(u_\lambda)(t)= \lambda^\frac{5-p}{p-1} E(u)(t)$ and hence the equation is energy-supercritical for $p>5$. Our goal is to show that given any (possibly large) initial data $(u_0, u_1)$, the supercritical nonlinear defocusing wave equation \eqref{eq:nlw} is globally well-posed at least for an open interval of exponents $p \in [5, 5 + \delta_0)$.

\begin{theorem}\label{thm:main} Let $\norm{(u_0, u_1)}_{\dot H^1 \cap \dot H^2 \times H^1} \leq M_0\, .$ Then there exists $\delta_0= \delta_0(M_0)>0$ such that for any $\delta \in (0, \delta_0)$ there exists a global solution $u$ of \eqref{eq:nlw} with $p=5+\delta$ from the initial data $(u_0, u_1)$. Moreover, there exists a universal constant $C> 1$ such that for any time $t$
\begin{equation}\label{eq:estimateonSobolev}
\norm{(u, \partial_t u)(t)}_{\dot H^1 \cap \dot H^2 \times H^1} \leq \norm{(u_0, u_1)}_{\dot H^1 \cap \dot H^2 \times H^1} e^{C(1+(CE(u))^{CE(u)^{352}})} 
\end{equation}
and we have the global spacetime bound
\begin{equation}\label{eq:estimateonspacetime}
\norm{u}_{L^{2(p-1)}(\mathbb{R}^3 \times \mathbb{R})} \leq C( 1+(CE(u))^{CE(u)^{352}} )\, .
\end{equation}
In particular, the solution scatters as $t \rightarrow \pm \infty\,.$
\end{theorem}


Global regularity and scattering for the energy-critical regime was established in \cite{Struwe2, Grillakis}. The classical results in the critical case were recently improved to obtain explicit double exponential bounds \cite{Tao2} and to allow a critical nonlinearity with an extra logarithmic factor $f(u)= u^5 \log(2+u^2)$ in the case of spherical symmetric data  \cite{Tao}. Exploiting the method introduced in \cite{Tao2}, \cite{Roy} could remove the assumption of spherical symmetry for slightly $ \log\log$-supercritical growth. In two-dimensions, global regularity has also been established for the slightly supercritical nonlinearity $f(u)=u e^{u^2}$ in \cite{Struwe3}. For the classical supercritical nonlinearity $f(u) = \lvert u \rvert^{p-1} u$ with $p>5$, global existence and scattering of solutions still holds for small data in scaling-invariant spaces, for instance in $\dot H^{s_p}\times \dot{H}^{s_p-1}$ where $$s_p:= 1 + \frac{\delta}{2(p-1)}$$ is the critical Sobolev exponent. For general large data however, the problem of global regularity and scattering is still open: Apart from  conditional regularity results in terms of the critical Sobolev regularity \cite{KenigMerle, KillipVisan}, global solutions have been built only from particular classes of initial data \cite{KriegerSchlag, BeceanuSoffer} or for a nonlinearity satisfying the null condition as in \cite{WangYu, MiaoPeiYu}. 

Our result should be seen in line with \cite{Tao2, Roy} pushing global regularity in a slightly supercritical regime. Although the nonlinearity considered in \cite{Tao2, Roy} has a logarithmically supercritical growth at infinity, it still comes, up to lower order terms, with the scaling associated to the critical case $p=5$. Correspondingly, both the scaling invariant quantities of the critical regime, as well as some logarithmically higher integrability, are controlled by the energy. Instead, we consider the supercritical nonlinearity \eqref{eq:nlw} and achieve global existence and scattering by paying the price of working on bounded sets of initial data, as previously done for other equations, such as SQG \cite{CotiVicol} and Navier-Stokes \cite{ColomboHaffter}. As in \cite{Roy, CotiVicol, ColomboHaffter}, the crucial ingredient of the proof of Theorem~\ref{thm:main} is a (quantitative) long-time estimate. In the spherically symmetric case, the classical Morawetz inequality gives an a priori spacetime bound as long as the solution exists. The following result replaces this long-time estimate in the absence of symmetry assumptions.

%
\begin{theorem}[A priori spacetime bound]\label{thm:spacetime} There exists universal constant $C \geq 1 $ such that  
for any solution $(u, \partial_t u) \in L^\infty (J, (\dot H^1 \cap \dot H^2 \times H^1)(\mathbb{R}^3))$ of \eqref{eq:nlw} with $p=5+\delta$, $\delta \in (0,1)$, denoting $M:=\norm{u}_{L^\infty(\mathbb{R}^3 \times J)}, \,E:=E(u) $ and $L:=\norm{(u, \partial_t u)}_{L^\infty(J, (\dot{H}^{s_p} \times \dot{H}^{{s_p}-1})(\mathbb{R}^3))} $ 
the following holds.

\begin{itemize}
\item if $\min \{ E M^\frac{\delta}{2}, L \} < c_0$, then $\norm{u}_{L^{2(p-1)}(\mathbb{R}^3 \times J)} \leq 1$
\item if $\min \{ E M^\frac{\delta}{2}, L \} \geq c_0$ and $(C EM^\frac{\delta}{2} L)^{C (EM^\frac{\delta}{2}L)^{176}} \leq 2^\frac{1}{\delta}$, then
\begin{align}\label{eq:finalest}
\norm{u}_{L^{2(p-1)}(\mathbb{R}^3 \times J)} \leq (C E M^\frac{\delta}{2} L)^{C (E M^\frac{\delta}{2}L)^{176}}\, .
\end{align}
\end{itemize}
\end{theorem}

\begin{corollary}\label{cor:spacetime} There exists a universal constant $C \geq 1$ such that the following holds. Let $M_0 >0$ given. Then there exists $\delta_0=\delta_0(M_0)>0$ such that for any solution $(u, \partial_t u) \in L^\infty(J, (\dot{H}^1 \cap \dot{H}^2 \times H^1)(\mathbb{R}^3))$ of \eqref{eq:nlw} with $p=5+\delta$ for $\delta \in (0, \delta_0]$ and with $\norm{(u, \partial_t u)}_{L^\infty(J, (\dot{H}^1 \cap \dot{H}^2 \times H^1)(\mathbb{R}^3))} \leq M_0$, we have the a priori spacetime bound
\begin{equation}\label{eq:doublexpM0}
\norm{u}_{L^{2(p-1)}(\mathbb{R}^3 \times J)} \leq \max \left \{ 1, (C E(u) M_0^\frac{\delta}{2})^{C (E(u) M_0^\frac{\delta}{2})^{352}}\right \} \, .
\end{equation}
\end{corollary}
\begin{remark} From the proof, we observe that $\delta_0$ has the following dependence as $M_0\to \infty$: There exists $C' \geq 1$ such that 
\begin{equation}
\delta_0:= \min \left \{1,\frac{\ln 2}{\ln M_0}, \frac{\ln 2}{\ln ( C' E) (C' E)^{352}} \right \} \,.
\end{equation}
\end{remark}
Theorem~\ref{thm:main} follows from Corollary~\ref{cor:spacetime} and a continuity argument, taking advantage of the fact that, if on one side the estimate \eqref{eq:doublexpM0} involves in the right-hand side higher order norms of the solution itself, which we a priori don't control for large times, on the other side they appear only to the power $\delta$ and hence can be kept under control for $\delta$ small. As regards the initial data, the statement of Theorem~\ref{thm:main} is written for simplicity with $(u_0, u_1) \in \dot H^1 \cap \dot H^2 \times H^1$ but a similar result would hold just above the critical threshold, namely for 
$(u_0, u_1) \in \dot H^1 \cap \dot H^{1+\eps} \times H^\eps$ for some $\eps >0$. Correspondingly, $\delta_0$ would also depend on $\eps$.

The proof of Theorem~\ref{thm:spacetime} follows instead the scheme introduced in \cite{Tao2} to obtain double exponential bounds on critical Strichartz norms based on Bourgain's ``induction on energy" method \cite{Bourgain}. In \cite{Roy}, the scheme has been successfully applied to a $\log$-supercritical equation assuming a (subcritical) a priori bound $M$ on $\norm{u}_{L^\infty(\R^3 \times J)}$: Indeed, it was noticed that the induction on the energy, which does not allow to include the a priori bound $M$, can actually be bypassed by a simpler ad-hoc argument. We will use the latter strategy also in our case. Rather than controlling a $L^4L^{12}$ norm as performed in the mentioned papers, we estimate $L^{2(p-1)}$ norm, which is scaling-critical for every $p$.
To follow their line of proof, we need to overcome some issues related to the supercritical nature of our equation:  For instance, a fundamental use of the equation in all critical global regularity results is the localized energy equality and the subsequent potential energy decay, first used in \cite{Struwe2, Grillakis, ShatahStruwe1}.  In the supercritical regime, the localized energy inequality becomes less powerful, since the nonlinear term is estimated this time in terms of a power of the length of the time interval besides the energy itself (see Lemma \ref{lem:potentialdecay}). To be able to still take advantage of this localized energy inequality, we need a control on the length of the so called unexceptional intervals which was not derived before in \cite{Tao2,Roy} and seems to work in the supercritical case only.
To achieve this control, we introduce another scaling invariant norm of $u$ accounting for more differentiability, namely $L^\infty \dot H^{s_p}$. This quantity, which appear in the final estimate \eqref{eq:finalest}, was not needed in \cite{Tao2,Roy}. It turns out fundamental to bound the length of unexceptional intervals by performing a mass concentration in $\dot H^{s_p}$, rather than in $\dot H^1$ (see Lemma \ref{lem:massconc}) and thereby obtaining an upper bound on the mass concentration radius.

The strategy of proof of Theorem \ref{thm:main} is very flexible and we plan to apply it in a future work to the radial supercritical Schr\"odinger equation.

\section{Preliminaries}\label{sec:prelim}
\subsection{Energy-flux equality}
With the notation of \cite{ShatahStruwe}, we introduce the forward-in-time wave cone, the truncated cone and their bounderies centered at $z_0=(x_0, t_0) \in \mathbb{R}^3 \times \mathbb{R}$ defined by
\begin{align}
K(z_0) &:= \{ z=(x,t) \in \mathbb{R}^4: \lvert x-x_0 \rvert \leq t-t_0 \} \, ,\\
K_s^t (z_0)&:= K(z_0) \cap (\mathbb{R}^3 \times [s,t]) \,, \\
M_s^t(z_0)&:= \{z=(x,r) \in \mathbb{R}^3 \times (s,t): \lvert x-x_0 \rvert = r- t_0 \} \,, \\
D(t; z_0)&:= K(z_0) \cap (\mathbb{R}^3 \times {t})\, .
\end{align}
 Correspondingly, we introduce the localized energy as well as the energy flux
\begin{align}
E(u; D(t; z_0)) &:=  \int_{D(t; z_0)} \frac{1}{2} \lvert \partial_t u \rvert^2 + \frac{1}{2} \lvert \nabla u \rvert^2 + \frac{\lvert u \rvert^{p+1}}{p+1}  \, \mathrm{d}x \, . \\
Flux(u, M_s^t(z_0)) &:= \int_{M_s^t(z_0)} \frac{1}{2} \left \lvert \nabla u - \frac{x-x_0}{\lvert x-x_0 \rvert} \partial_t u \right \rvert^2 + \frac{\lvert u \rvert^{p+1}}{p+1} \frac{\mathrm{d}\sigma}{\sqrt{2}} \, .
\end{align}
Let us recall, that for any sufficiently regular solution we have the energy-flux identity
\begin{equation}\label{eq:localizedenergy1}
E(u; D(t; z_0)) + Flux(u; M_s^t(z_0))= E(u; D(s; z_0)) \,
\end{equation}
for any $0<s<t\,.$Indeed, \eqref{eq:localizedenergy1} is obtained by integration of $(\Box u - \lvert u \rvert^{p-1} u ) \partial_t u $ on $K_s^t(z_0)$, see for instance \cite{ShatahStruwe}. Whenever $z_0=(0,0)$, we will not write the dependence on $z_0$, we will write $\Gamma_+(I)$ for the forward wave cone centered in $0$ and truncated by $I$
$$\Gamma_+(I):= \{(x,t) \in \R^3 \times \R: |x|<t, \, t\in I\},$$ and we denote $e(t):=E(u; D(t)) \, $. We can then rewrite \eqref{eq:localizedenergy1} for any $0<s<t$
\begin{equation}
e(t)-e(s)= \int_{M_s^t} \frac{1}{2} \lvert \nabla u - \frac{x}{t} \partial_t u \rvert^2 + \frac{\lvert u \rvert^{p+1}}{p+1} \, \frac{\mathrm{d}\sigma}{2} \, .
\end{equation}


\subsection{Strichartz estimates} Let $u: \mathbb{R}^3 \times I \to \mathbb{R}$ solve the linear wave equation $\Box u = F$. Let $m\in[ 1,3/2)$. Then for any $(q,r)\in (2, \infty] \times [1, \infty)$ wave-$m$-admissible and  for any conjugate pair $(\tilde{q}, \tilde{r}) \in [1, +\infty]\times [1, + \infty]$ with
\begin{equation}\label{eq:mwaveadmissible}
\frac{1}{\tilde{q}} + \frac{3}{\tilde{r}}-2=\frac{1}{q}+\frac{3}{r}=\frac{3}{2}-m
\end{equation}
we have
\begin{equation}\label{eq:mstrichartz}
\norm{u}_{L^q(I, L^r)} + \norm{(u, \partial_t u)}_{L^\infty(I, \dot{H}^m \times \dot{H}^{m-1})} \leq C\left( \norm{(u,  \partial_t u)(t_0)}_{\dot{H}^m_x \times \dot{H}^{m-1}_x} +\norm{F}_{L^{\tilde{q}}(I, L^{\tilde{r}})} \right) \, ,
\end{equation}
where $t_0 \in I$ is a generic time. Notice that $(q,r)=(2(p-1), 2(p-1))$ is wave-$s_p$- admissible and all $(q,r)$ wave-$s_p$-admissible are scaling-critical.
Moreover, the constant $C$ can be taken independent on $m \in [1,5/4]$.

\subsection{Localized Strichartz estimates} By the finite speed of propagation, we can localize the above Strichartz estimates on wave cones. Let $I=[a,b]$ and $m \in [ 1, \frac 32)$. For any solution $u: \mathbb{R}^3 \times I \rightarrow \mathbb{R}$ of a linear wave equation $\Box u = F$, we have for any $(q, r)$ wave-$m$-admissible and any conjugate pair $(\tilde{q}, \tilde{r})$ satisfying \eqref{eq:mwaveadmissible} the localized estimate
\begin{align}\label{eq:mstrichartzlocalized}
\norm{u}_{L^q L^r (\Gamma_+(I))} &\lesssim \norm{(u, \partial_t u)(b)}_{(\dot{H}^m \times \dot{H}^{m-1})(\mathbb{R}^3)}+ \norm{F}_{L^{\tilde{q}} L^{\tilde{r}}(\Gamma_+(I))} \, .
\end{align}
As a consequence, if $I= [a,b]=J_1 \cup J_2$, we have
\begin{align}
\label{eq:mstrichartzlocalized2}
\norm{u}_{L^q L^r (\Gamma_+(J_1))} &\lesssim \norm{(u, \partial_t u)(b)}_{(\dot{H}^m \times \dot{H}^{m-1})(\mathbb{R}^3)}+ \norm{F}_{L^{\tilde{q}} L^{\tilde{r}}(\Gamma_+(J_1 \cup J_2))} \, .
\end{align}

\subsection{Littlewood-Paley projection}
We follow the presentation of \cite{Tao3}. Fix $\phi \in C^\infty_c(\mathbb{R}^d)$ radially symmetric, $0\leq \phi \leq 1$ such that $\supp \phi \subseteq B_2(0)$ and $\phi \equiv 1$ on $B_1(0)$. For $N \in 2^\mathbb{Z}$, introduce the Fourier multipliers
\begin{align}
\widehat{P_{\leq N } f}(\xi)&:= \phi(\xi/N) \hat{f}(\xi) \,, \\
\widehat{P_{> N} f}(\xi)&:= (1-\phi(\xi/N))\hat{f}(\xi) \,, \\
\widehat{P_N f}(\xi)&:= (\phi(\xi/N) - \phi(2\xi/N) ) \hat{f}(\xi) \, .
\end{align}
The above projections can equivalently be written as convolution operators and Young inequality shows that the Littlewood-Paley projections are bounded on $L^p$ for any $1\leq p \leq + \infty\, .$ Moreover, we have the Bernstein's inequalities
\begin{align}\label{eq:Bernstein}
\norm{P_{\leq N} f}_{L^{q}_x(\mathbb{R}^d)}  &\lesssim_{p,q} N^{d(\frac 1p- \frac 1q)} \norm{P_{\leq N} f}_{L^p_x (\mathbb{R}^d)} 
\end{align}
for $1 \leq p \leq q \leq + \infty$ and the same holds with $P_ N f$ in place of $P_{\leq N} f$.
Moreover, for $1< p< + \infty$ we also recall the fundamental Paley-Littlewood inequality
\begin{equation}\label{eq:PaleyLittlewoodineq}
\norm{f}_{L^p(\mathbb{R}^d)} \sim  \norm{( \sum_{N \in 2^\mathbb{Z}} \lvert P_N f \rvert^2 )^\frac{1}{2}}_{L^p(\mathbb{R}^d)}.
\end{equation} 


\subsection{Dependence of constants}In the rest of the paper, all constants will be independent on the choice of $\delta \in [0,1)$. We keep the estimates in scaling invariant form (for instance, in all the statements of the Lemmas in Sections~\ref{sec:small}-~\ref{sec:rev-mass}). We write the terms in the estimate in terms of simpler scaling invariant quantities, such as $E \|u \|_{L^{\infty}}^{\delta/2}$, $\|u \|_{L^{2(p-1)}}$,  $\|u \|_{L^\infty \dot H^{s_p}}$, $ET^{-\frac{\delta}{p-1}}$ (see for instance \eqref{eq:lowerboundpotential}).

\section{Spacetime norm bound under a scaling invariant smallness assumption 
}\label{sec:small}
We recall that the nonlinear wave equation has bounded $L^{2(p-1)}$ norm if we assume a suitable smallness on the solution, which must be in terms of scaling invariant quantities. We will need it in terms of the critical $\dot H^{s_p}$ norm as well as a combination of the energy and the $L^\infty$ norm.
\begin{lemma}\label{lem:smallenergy} Let $p=5+\delta$ for $\delta \in (0,1)$ and consider a solution $(u, \partial_t u) \in L^\infty (I, \dot H^1 \cap \dot H^2 \times  H^1)$ to \eqref{eq:nlw}. Assume additionally that $\norm{u}_{L^\infty(\mathbb{R}^3\times I)} \leq M$. There exists a universal $0<c_0<1$ such that if 
\begin{equation}
EM^\frac{\delta}{2}  \leq c_0 \text{ or } \norm{(u, \partial_t u)}_{L^\infty (I, (\dot H^{s_p} \times \dot H^{s_p-1})(\mathbb{R}^3)) }\leq c_0 \,,
\end{equation} 
then 
\begin{equation}\label{eqn:small-en-ts}
\norm{u}_{L^{2(p-1)}(\mathbb{R}^3 \times I)}\leq 1 \, .
\end{equation}
\end{lemma}
\begin{proof} Let us first assume that $EM^\frac{\delta}{2} \leq c_0$ for a $c_0<1$ yet to be chosen. By interpolation 
\begin{equation}
\norm{u}_{L^{2(p-1)}} \leq \norm{u}_{L^\infty}^\frac{\delta}{p-1} \norm{u}_{L^8}^\frac{4}{p-1}  \, .
\end{equation}
We notice that $(8,8)$ is wave-1-admissible. By Strichartz \eqref{eq:mstrichartz} (with $m=1$ and $(\tilde{q}, \tilde{r})=(2, \frac 3 2)$), H\"older and the Sobolev embedding $\dot{H}^1(\R^3) \hookrightarrow L^6(\R^3)$ we have
\begin{equation} 
\norm{u}_{L^8_{t,x}} \lesssim E^\frac{1}{2} + \norm{\vert u \rvert^{p-1} u}_{L^2 L^{3/2}}  \lesssim E^\frac{1}{2} + \norm{\lvert u \rvert^{p-1}}_{L^{2}_{t,x}} \norm{u}_{L^\infty L^6} \lesssim E^\frac{1}{2} \left(1 + \norm{u}_{L^{2(p-1)}}^{p-1}\right) \, .
\end{equation}
Summarizing, we have obtained that for a $C\geq1$
\begin{equation}
\norm{u}_{L^{2(p-1)}} \leq C (M^\frac{\delta}{2} E)^\frac{2}{p-1} (1 + \norm{u}_{L^{2(p-1)}}^4) \, ,
\end{equation}
from which \eqref{eqn:small-en-ts} follows setting $c_0:=(4C)^{-\frac{p-1}{2}} <1
$
. 

Let us now assume that $\norm{(u, \partial_t u)}_{L^\infty (\dot H^{s_p} \times \dot H^{s_p-1})} \leq c_0'$ for a  $0<c_0'<1$. Observing that $(2(p-1), 2(p-1))$ is wave-$s_p$-admissible, we have by Strichartz \eqref{eq:mstrichartz} (with $m=s_p$ and {$(\tilde{q}, \tilde{r})=\left(2, \frac{6(p-1)}{3p+1}\right)$}), H\"older and the Sobolev embedding $\dot{H}^{s_p}(\mathbb{R}^3) \hookrightarrow L^\frac{3(p-1)}{2}(\mathbb{R}^3)$ 
\begin{align}
\norm{u}_{L^{2(p-1)}} &\lesssim \norm{(u, \partial_t u)}_{L^\infty (\dot{H}^{s_p} \times \dot H^{s_p-1})} + \norm{\lvert u \rvert^{p-1} u}_{L^2 L^{6(p-1)/(3p+1)}} \\
&  \lesssim \norm{(u, \partial_t u)}_{L^\infty( \dot{H}^{s_p} \times \dot H^{s_p-1})} + \norm{\lvert u \rvert^{p-1}}_{L^2_{t,x}} \norm{u}_{L^\infty L^{3(p-1)/2}} \\
&  \lesssim \norm{(u, \partial_t u)}_{L^\infty (\dot{H}^{s_p} \times \dot H^{s_p-1})}(1+\norm{ u}_{L^{2(p-1)}}^{p-1} ) \, . 
\end{align}
Calling $C'$ the constant in the above inequality, \eqref{eqn:small-en-ts} follows by setting $c_0':=(4C')^{-1}\,.$
\end{proof}

\section{Spacetime norm decay in forward wave cones}\label{sec:forward}

The goal of this section is to prove the following proposition, which individuates a subinterval $J$ (of quantified length) with small $L^{2(p-1)}$ norm of $u$ in any sufficiently large given interval $I=[T_1, T_2]$. The main difference to the energy-critical case $p=5$ \cite[Corollary 4.11]{Tao2} lies in the fact that the largeness requirement on $I$ can no longer be reached by simply choosing $T_2$ big enough (see Remark \ref{rem:T2}).

\begin{prop}[Spacetime-norm decay]\label{cor:spacetimedecay} Let $p=5+\delta$ with $\delta \in (0,1)$, $I=[T_1,T_2] \subset (0,\infty)$ and consider a solution $(u, \partial_t u) \in L^\infty (I, \dot H^1 \cap \dot H^2 \times  H^1)$ to \eqref{eq:nlw}. Assume that $\norm{u}_{L^\infty(\mathbb{R}^3 \times I)} \leq M$. There exists a universal constant $0<C_2<1$ such that if $0<\eta< 1$ is such that 
\begin{align}\label{eq:spacetimedecayhyp1}
\eta < C_2  (EM^\frac{\delta}{2})^\frac{7}{6(p-1)}
\end{align}
then the following holds for any $A$ satisfying
\begin{equation}\label{eq:spacetimedecayhypA}
A > (C_2 \eta^{-1})^\frac{12(p-1)}{5}(EM^\frac{\delta}{2})^\frac{14}{5} \, :
\end{equation}If $T_1$ and $T_2$ are such that 
\begin{equation}\label{eq:spacetimedecayhypI}
\frac{T_2}{T_1} \geq A^{3 (C_2 \eta^{-1})^\frac{6(p-1)(p+1)}{5} (EM^\frac{\delta}{2})^\frac{9p+19}{10}\max\{(C_2 \eta^{-1})^\frac{-6(p-1)^2}{5}  (EM^\frac{\delta}{2})^\frac{9(p-1)}{10} ,(M^\frac{p-1}{2} T_2)^\frac{\delta}{2} \} }  \, ,
\end{equation}
 then there exists a subinterval $J=[t', At'] \subseteq I$ with
\begin{equation}\label{eqn:spacetimethesis}
\norm{u}_{L^{2(p-1)}(\Gamma_+(J))} \leq \eta \, .
\end{equation}
\end{prop}
\begin{remark}[Simplified assumptions in the large energy regime]\label{rmk:spacetimedecayhyp3}In the large energy regime $E M^\frac{\delta}{2} \geq c_0$, with $c_0$ defined through Lemma \ref{lem:smallenergy}, the hypothesis \eqref{eq:spacetimedecayhyp1} can be simplified to
\begin{equation}\label{eq:spacetimedecayhyp3}
\eta <C_2 c_0^\frac{7}{6(p-1)}:=  c_0' \, ,
\end{equation}
where we observe that $0<
c_0' \leq 1$. Moreover, the assumption \eqref{eq:spacetimedecayhypI} can be replaced by the stronger condition
\begin{equation}\label{eq:spacetimedecayhypI-simple}
\frac{T_2}{T_1} \geq A^{3 (C_2 \eta^{-1})^\frac{6(p-1)(p+1)}{5} (EM^\frac{\delta}{2})^\frac{9p+19}{10}\max\{c_0^\frac{p-1}{2}, (M^\frac{p-1}{2} T_2)^\frac{\delta}{2} \} }  \, .
\end{equation}

\end{remark}

\begin{remark}\label{rem:T2}
The assumptions of Proposition~\ref{cor:spacetimedecay} are clearly verified as an upper bound on $T_1$ for any fixed $\eta$ satisfying \eqref{eq:spacetimedecayhyp1}, $A$ satisfying \eqref{eq:spacetimedecayhypA} and $T_2$ satisfying \eqref{eq:spacetimedecayhypI}. 
However this will not be the spirit of the application of this Proposition: we will rather fix $T_1$ and consider \eqref{eq:spacetimedecayhypI} as a condition on $T_2$ and $\delta$. This condition may sound strange since, when all other parameters are fixed, \eqref{eq:spacetimedecayhypI} is not verified for large $T_2$. On the other side, we will instead fix $T_2:= T_1A^{3 (C_2 \eta^{-1})^\frac{6(p-1)(p+1)}{5} (EM^\frac{\delta}{2})^\frac{9p+19}{10}}$ and notice that \eqref{eq:spacetimedecayhypI} is verified for $\delta$ sufficiently small.
\end{remark}

As a first step to the proof of Proposition~\ref{cor:spacetimedecay}, we show that if the $L^{2(p-1)}$ norm of $u$ in a strip is bounded from below, the Strichartz estimates imply a lower bound on the $L^\infty L^{p+1} $ norm in the same interval.
\begin{lemma}[Lower bound on global and local potential energy]\label{lem:lowerboundpotential} Let $p=5+\delta$ with $\delta \in (0,1)$ and $\eta \in (0,1]$. Consider a solution $(u, \partial_t u) \in L^\infty (I, \dot H^1 \cap \dot H^2 \times  H^1)$ to \eqref{eq:nlw}. Assume that $\norm{u}_{L^{2(p-1)}(\mathbb{R}^3 \times I)} \geq \eta$ and $\norm{u}_{L^\infty(\mathbb{R}^3 \times I)} \leq M$. Then there exists $0< C_1 \leq 1$ universal such that 
\begin{equation}\label{eq:lowerboundpotential}
\norm{u}_{L^\infty(I, L^{p+1})}^{p+1} \geq C_1 \eta^{\frac{12}{5}(p-1)}(M^\frac{\delta}{2}E)^{-\frac{9}{5}} M^{-\frac{\delta}{2}} \, .
\end{equation}
Moreover, by finite speed of propagation the same estimate can be obtained by replacing $\mathbb{R}^3 \times I$ by any truncated forward wave cone $\Gamma_+(I)$.
\end{lemma}
\begin{proof}
Let $0< \eta \leq 1$. By shrinking $I$, we can assume w.l.o.g. that $\norm{u}_{L^{2(p-1)}(\mathbb{R}^3 \times I)}= \eta$. 
 We observe that we control all wave-1-admissible spacetime norms with the energy. Indeed, fix $(q,r)$ wave-1-admissible.
  By the Strichartz estimate \eqref{eq:mstrichartz} with $m=1$ and H\"older
\begin{equation}\label{eq:bla}
\norm{u}_{L^q L^r} \lesssim E^\frac{1}{2}+ \norm{\lvert u \rvert^{p-1}u}_{L^2L^{3/2}} \lesssim E^\frac{1}{2} +   \norm{u}_{L^\infty L^6} \norm{\lvert u \rvert^{p-1}}_{L^2_{t,x} }\lesssim E^\frac{1}{2} + E^\frac{1}{2} \eta^{p-1} \lesssim E^\frac{1}{2} \, .
\end{equation}
We observe that the pair $(3, 18)$ is wave-1-admissible and that $(3, 18)$ and $(\infty, p+1)$ interpolate to $(\left(\frac56 (p+1)+3, \frac56 (p+1)+3\right)=(8+\frac56 \delta, 8+\frac56 \delta)$. By interpolation and \eqref{eq:bla}, we thus have
\begin{align}
\norm{u}_{L^{2(p-1)}}^{2(p-1)} &\leq \norm{u}_{L^\infty_{t,x}}^{\frac{7}{6}\delta} \norm{u}_{L^{8+\frac{5}{6}\delta}_{t,x}}^{8+ \frac{5}{6}\delta} \leq M^{\frac{7}{6}\delta} \norm{u}_{L^\infty L^{p+1}}^{\frac{5}{6}(p+1)} \norm{u}_{L^3 L^{18} }^3 \lesssim (M^\frac{\delta}{2} E)^\frac{3}{2} M^{\frac{5}{12}\delta}  \norm{u}_{L^\infty  L^{p+1}}^{\frac{5}{6}(p+1)} \, . \qedhere
\end{align}
\end{proof}

We now come to a localized energy inequality of Morawetz-type which, in the critical case $p=5$, implies the potential energy decay and hence it is crucial for the global regularity in the critical case \cite{Grillakis, Struwe2}. In the supercritical case, the former localized energy inequality degenerates and will only lead to some decay estimate on bounded intervals: indeed the presence of the extra term $ b^\frac{\delta}{p+1}$ in the right-hand side of \eqref{eq:potentaildecay} below makes the inequality interesting only when an estimate on the length of the interval is at hand. 

\begin{lemma}\label{lem:potentialdecay} Let $\delta \in [0,1)$ and $p=5+\delta$. For any $0<a<b$ and any weak finite energy solution $(u, \partial_t u) \in C([a,b], \dot{H}^1\cap L^{p+1})\cap L^p([a,b], L^{2p}) \times C([a,b], L^2)$ of \eqref{eq:nlw}, we have 
\begin{equation}\label{eq:potentaildecay}
\int_{\lvert x \rvert \leq b} \lvert u(x, b) \rvert^{p+1} \, \mathrm{d}x \lesssim  \frac{a}{b}E + e(b)-e(a)+ b^\frac{\delta}{p+1} (e(b)-e(a))^\frac{2}{p+1}  \, .
\end{equation}
\end{lemma}
\begin{proof} 
Let us first assume that $u \in C^2(\R^3 \times [a,b])$ is a classical solution of \eqref{eq:nlw}. We follow the notation of \cite{ShatahStruwe1, BahouriShatah} and introduce the quantities
\begin{align}
Q_0 &:= \frac{1}{2} \left( (\partial_t u )^2 + \lvert \nabla u \rvert^2 \right)+ \frac{\lvert u \rvert^{p+1} }{p+1} +\partial_t u \left( \frac{x}{t} \cdot \nabla u \right) \, \\
P_0 &:=\frac{x}{t} \left( \frac{(\partial_t u )^2}{2}-\frac{\lvert \nabla u \rvert^2}{2}- \frac{\lvert u \rvert^{p+1}}{p+1}\right) + \nabla u \left( \partial_t u+ \frac{x}{t} \cdot \nabla u + \frac{u}{t} \right) \, \\
R_0 &:= \left(1-\frac{4}{p+1} \right) \lvert u \rvert^{p+1} \, .
\end{align}
Observe $R_0 \geq 0 \, .$ Multiplying \eqref{eq:nlw} by $(t \, \partial_t u + x \cdot \nabla u + u)$ one obtains
$\partial_t(t\, Q_0+ \partial_t u \, u) - \div(tP_0)+R_0=0 \,,$
see \cite[Chapter 2.3]{ShatahStruwe}. Integrating on $K_a^b$ (recall the definitions in Section~\ref{sec:prelim}), we obtain
\begin{align}\label{eq:Q0int}
b \int_{D(b)} Q_0 \, dx- a\int_{D(a)} &Q_0 \, dx + \int_{K_a^b} R_0 \,dx \, dt \nonumber \\
&= -\int_{D(b)} \partial_t u u \, dx+ \int_{D(a)} \partial_t u u \, dx+ \int_{M_a^b} \left(t \,Q_0+ \partial_t u \,u + t P_0 \cdot \frac{x}{\lvert x \rvert} \right) \frac{\mathrm{d}\sigma}{\sqrt{2}} \nonumber
\\
 &=\int_{M_a^b} t\left(\partial_t u + \frac  x t \cdot \nabla u + \frac{u}{t} \right)^2 \frac{\mathrm{d}\sigma}{\sqrt{2}} \, ,
\end{align}
where in the second equality we used the computations of \cite[Section 2]{BahouriShatah} for $p=5$ to rewrite the last addendum on the right-hand side. Indeed, on $M_a^b$ the integrand $t \, Q_0+ \partial_t u + P_0 \cdot  \frac{x}{\lvert x \rvert} = t (\partial_t u )^2 + 2 \partial_t u x \cdot \nabla u + \partial_t u \,, u$ is now independent of $p\,$. Proceeding as \cite{BahouriGerard}, we estimate on  $K_a^b$
\begin{equation}\label{eq:Q1}
\partial_t u \frac{x}{t} \cdot \nabla u  \leq \frac{(\partial_t u)^2}{2} + \frac{1}{2} \left \lvert \frac{x}{t} \cdot \nabla u \right \rvert^2 \leq \frac{(\partial_t u)^2}{2} + \frac{1}{2} \lvert \nabla u \rvert^2 \, .
\end{equation} 
We infer from  \eqref{eq:Q0int}-\eqref{eq:Q1}, the positivity of $R_0$ and the conservation of the energy that 
\begin{align}
\int_{D(b)} \frac{\lvert u \rvert^{p+1}}{p+1} \, dx &\leq \frac{a}{b} \int_{D(a)} Q_0 \, dx + \frac{1}{b} \int_{M_a^b} t \left( \partial_t u + \frac{x}{t} \cdot \nabla u + \frac{u}{t} \right)^2 \, \frac{\mathrm{d} \sigma }{\sqrt{2}}  \\
&\leq \frac{a}{b} \int_{D(a)} \left( \frac{\lvert u \rvert^{p+1}}{p+1} + (\partial_t u )^2 + \lvert \nabla u \rvert^2 \right) \, dx+ \frac{1}{b} \int_{M_a^b} t \left( \partial_t u + \frac{x}{t} \cdot \nabla u + \frac{u}{t} \right)^2 \, \frac{\mathrm{d} \sigma }{\sqrt{2}} \\
&\leq \frac{a}{b} E + \frac{1}{b} \int_{M_a^b} t \left( \partial_t u + \frac{x}{t} \cdot \nabla u + \frac{u}{t} \right)^2 \, \frac{\mathrm{d} \sigma }{\sqrt{2}}
 \, .
\end{align}
The last term on the right-hand side we estimate as in \cite{BahouriGerard}: We use \eqref{eq:localizedenergy1} to bound
\begin{align}
\frac{1}{b} \int_{M_a^b} t \left( \partial_t u + \frac{x}{t} \cdot \nabla u + \frac{u}{t} \right)^2 \, \frac{\mathrm{d} \sigma }{\sqrt{2}} 
\leq2(e(b)-e(a))  + 2 \int_{M_a^b} \frac{u^2}{t^2} \, \frac{\mathrm{d}\sigma}{\sqrt{2}}
 \, .
\end{align}
The main difference with respect to the energy-critical regime is the  estimate of the second addendum which now deteriorates with $b \, .$ Indeed, we estimate by H\"older
\begin{align}
\int_{M_a^b} \frac{u^2}{t^2} \, \frac{\mathrm{d}\sigma}{\sqrt{2}}  \leq b^\frac{\delta}{p+1} \left( \int_{M_a^b} \frac{\lvert u \rvert^{p+1}}{p+1} \frac{\mathrm{d}\sigma}{\sqrt{2}} \right)^\frac{2}{p+1} \lesssim b^\frac{\delta}{p+1} \left( e(b)-e(a) \right)^\frac{2}{p+1} \,,
\end{align}
Collecting terms, we have obtained \eqref{eq:potentaildecay} for classical solutions $u \in C^2(\R^3 \times [a,b]) \, .$

If $u$ is a weak finite energy solution of \eqref{eq:nlw} as in the statement, we proceed as in \cite{BahouriGerard}: we fix a family of mollifiers $\{\rho_\epsilon\}_{\epsilon>0}$ in space and define $u_\epsilon:= u \ast \rho_\epsilon\,.$ Then, setting $f_\epsilon= - \lvert u_\epsilon \rvert^{p-1} u_\epsilon + (\lvert u \rvert^{p-1} u) \ast  \rho_\epsilon\,$, $u_\epsilon \in C^2(\R^3 \times [a,b])$ is a classical solution of
\begin{equation}\label{eq:nlwrhs}
\Box u_\epsilon= \lvert u_\epsilon \rvert^{p-1} u_\epsilon + f_\epsilon\,.
\end{equation}
By assumption, $f_\epsilon \in L^1([a,b], L^2)$ can be treated as a source term. We then deduce \eqref{eq:potentaildecay} by proving the analogous local energy inequality for a nonlinear wave equation with right-hand side \eqref{eq:nlwrhs} and pass to the limit $\epsilon \to 0 \, .$ We refer to \cite[Lemma 2.3]{BahouriGerard} for details.
\end{proof}

Lemma \ref{lem:potentialdecay} can be viewed as decay estimate for the potential energy. Again, when compared to the critical case \cite[Corollary 4.10]{Tao2}, the supercriticality of the equation weakens the decay by introducing a new dependence on $T_2$, the endpoint of the interval to which the decay estimate is applied, which deteriorates as $T_2 \to + \infty$.

\begin{prop}[Potential energy decay in forward wave cones]\label{prop:potenergydecay}Let $I=[T_1, T_2]\subset (0, +\infty)$ and consider a solution $(u, \partial_t u) \in L^\infty (I, \dot H^1 \cap \dot H^2 \times  H^1)$ to \eqref{eq:nlw} with $p=5+\delta$ for some $\delta \in (0,1)$. Let $0< \theta$  such that 
\begin{equation}\label{eq:potenergydecyhypeta}
ET_2^{-\frac{\delta}{p-1}} \theta^{-(p+1)} > 1 \, .
\end{equation}
Let $A>0$ be such that
\begin{equation}\label{eqn:req-A}
A \geq ET_2^{-\frac{\delta}{p-1}} \theta^{-(p+1)} \qquad \mbox{and} \qquad A^{3 E T_2^{-\frac{\delta}{p-1}} \theta^{-(p+1)}\max\{1, \theta^{-\frac{(p+1)(p-1)}{2}} \}}T_1\leq T_2,
\end{equation}
 then there exists a subinterval of the form $J=[t', At']$ such that 
\begin{equation}
\norm{u}_{L^\infty L^{p+1}(\Gamma_+(J))} \lesssim T_2^\frac{\delta}{(p-1)(p+1)} \theta \,.
\end{equation}
\end{prop}
Notice that $\theta$ in the previous statement is not dimensional.

\begin{proof}
Let $\theta>0$ be as in \eqref{eq:potenergydecyhypeta} and fix $A\geq ET_2^{-\frac{\delta}{p-1}} \theta^{-(p+1)}.$ Let $N$ to be chosen later be such that $A^{2N} T_1\leq T_2$, namely
$$ \bigcup_{i=1}^N [A^{2(n-1)}T_1, A^{2n} T_1] \subseteq I \, .$$ Since $e$ is non-decreasing in time (see \eqref{eq:localizedenergy1}), we have $e(A^{2n}t)- e(A^{2(n-1)}t) \geq 0$ for all $n$ and
\begin{equation}
0 \leq \sum_{n=1}^N e(A^{2n}T_1)- e(A^{2(n-1)}T_1)= e(A^{2N} T_1)- e(T_1) \leq E \,.
\end{equation}
Hence there exists $n_0 \in \{1, \dots, N \}$ such that 
$e(A^{2n_0}T_1)- e(A^{2(n_0-1)}T_1) \leq E N^{-1} \, .$
Splitting the interval $[A^{2(n_0-1)}T_1, A^{2n_0}T_1] = [A^{2(n_0-1)}T_1, A^{2n_0-1}T_1] \cup [A^{2n_0-1}T_1, A^{2n_0}T_1]$ we have, applying Lemma \ref{lem:potentialdecay} with $a:= A^{2(n_0-1)}T_1$ and varying $b \in [A^{2n_0-1}T_1, A^{2n_0}T_1]$, that
\begin{align}
\norm{u}_{L^\infty L^{p+1}(\Gamma_+([A^{2n_0-1}T_1, A^{2n_0}T_1])}^{p+1} &\lesssim \frac{1}{A}E+ E N^{-1}+ (A^{2n_0} T_1)^\frac{\delta}{p+1}(E N^{-1})^\frac{2}{p+1} \\
&\lesssim T_2^\frac{\delta}{p-1} \theta^{p+1} + E N^{-1}+ T_2^\frac{\delta}{p+1}(E N^{-1})^\frac{2}{p+1} \\
&\lesssim T_2^\frac{\delta}{p-1} \theta^{p+1} \, ,
\end{align}
provided $(EN^{-1})^\frac{2}{p+1} \leq T_2^\frac{2\delta}{(p-1)(p+1)} \theta^{p+1} $ and $EN^{-1} \leq T_2^\frac{\delta}{p-1} \theta^{p+1} $, or equivalently, 
\begin{equation}
E T_2^{-\frac{\delta}{p-1}} \theta^{-(p+1)}\max\{1, \theta^{-\frac{(p+1)(p-1)}{2}} \}  \leq N \, .
\end{equation} 
For the latter, we have to ask that  $[T_1, A^{2N}T_1] \subseteq [T_1, T_2]$, which is enforced by the second requirement in \eqref{eqn:req-A}.
\end{proof}

\begin{proof}[Proof of Proposition~\ref{cor:spacetimedecay}] Fix $0<\theta $ yet to be determined such that $ET_2^{-\frac{\delta}{p-1}} \theta^{-(p+1)} > 1$. Fix $A\geq ET_2^{-\frac{\delta}{p-1}} \theta^{-(p+1)}$ and assume that \eqref{eqn:req-A} holds. By Proposition \ref{prop:potenergydecay}, there exists a subinterval $J$ of the form $J:=[t', At']$ and $C'\geq 1$ such that
\begin{equation}\label{eqn:august}
\norm{u}_{L^\infty L^{p+1}(\Gamma_+(J))} \leq C' T_2^\frac{\delta}{(p-1)(p+1)} \theta \, .
\end{equation}
We claim that if we choose $\theta$ appropriately, we have $\norm{u}_{L^{2(p-1)}(\Gamma_+(J))} \leq \eta$. Indeed, assume by contradiction that $\norm{u}_{L^{2(p-1)}(\Gamma_+(J))} \geq \eta$. Then we have from Lemma \ref{lem:lowerboundpotential} 
\begin{equation}
\norm{u}_{L^\infty L^{p+1} (\Gamma_+(J))} \geq C_1 \eta^{\frac{12(p-1)}{5(p+1)}} (M^\frac{\delta}{2}E)^{-\frac{9}{5(p+1)}} M^{-\frac{\delta}{2(p+1)}} \, .
\end{equation}
Choosing $\theta$ to be 
\begin{equation}
\theta := \frac{C_1}{2 C'} \eta^{\frac{12(p-1)}{5(p+1)}} (M^\frac{\delta}{2}E)^{-\frac{9}{5(p+1)}} M^{-\frac{\delta}{2(p+1)}} T_2^{-\frac{\delta}{(p+1)(p-1)}} \,,
\end{equation}
we reach a contradiction with \eqref{eqn:august}. Let us now verify the hypothesis on $\theta$: We observe that  $$ET_2^{-\frac{\delta}{p-1}} \theta^{-(p+1)} = (C_1 (2C')^{-1})^{-(p+1)} \eta^{-\frac{12(p-1)}{5}} (EM^\frac{\delta}{2})^\frac{14}{5} \,,$$
such that hypothesis \eqref{eq:potenergydecyhypeta} is enforced, if
$$ 0<\eta < (C_1^{-1}2C')^{\frac{5(p+1)}{12(p-1)}} (EM^\frac{\delta}{2})^\frac{7}{6(p-1)} \, . $$ This explains the hypothesis \eqref{eq:spacetimedecayhyp1} and \eqref{eq:spacetimedecayhypA} with the choice $C_2:=  (C_1^{-1}2C')^{\frac{5(p+1)}{12(p-1)}}$. We also rewrite the largeness hypothesis on $I$, namely the second formula in \eqref{eqn:req-A}, in terms of $\eta$
\begin{align}
\theta^{-(p+1)(p-1)/2} &= (C_1 (2C')^{-1})^{-(p+1)(p-1)/2}  \eta^{-\frac{6(p-1)^2}{5}} (EM^\frac{\delta}{2})^\frac{9(p-1)}{10}  M^\frac{\delta(p-1)}{4}T_2^\frac{\delta}{2} 
\\ &= (C_2 \eta^{-1})^\frac{6(p-1)^2}{5}  (M^\frac{p-1}{2}T_2)^\frac{\delta}{2} (EM^\frac{\delta}{2})^\frac{9(p-1)}{10}
  \, ,
\end{align}
so that
\begin{align}
\max\{ 1, \theta^{-(p+1)(p-1)/2} \} &=   (C_2 \eta^{-1})^\frac{6(p-1)^2}{5} (EM^\frac{\delta}{2})^\frac{9(p-1)}{10} \max\{  (C_2 \eta^{-1})^\frac{-6(p-1)^2}{5} (EM^\frac{\delta}{2})^\frac{-9(p-1)}{10} ,   (M^\frac{p-1}{2}T_2)^\frac{\delta}{2} \}
  \,.
\end{align}
This shows that \eqref{eq:spacetimedecayhypI} implies the second inequality in \eqref{eqn:req-A}.
\end{proof}

\section{Asymptotic stability}\label{sec:asympt}
Let $u: \mathbb{R}^3 \times I \rightarrow \R$ solve an inhomogeneous wave equation $\Box u = F$.  We now introduce the free evolution $u_{l, t_0}$ from time $t_0$, that is the unique solution of the free wave equation $\Box u_{l, t_0}=0$ which agrees with $u$ at time $t_0$, that is $(u_{l, t_0}, \partial_t u_{l, t_0})(t_0)=(u, \partial_t u)(t_0) \, .$ We recall that, from solving the linear wave equation in Fourier space, we have  the representation formula 
\begin{equation}\label{eq:lineareqrepfromula}
u_{l, t_0}(t) = \cos (t \sqrt{-\Delta}) u(t_0)+ \frac{\sin(t \sqrt{-\Delta})}{\sqrt{-\Delta}} \partial_t u(t_0) \, ,
\end{equation}
where we use Fourier multiplier notation (see for instance \cite{Sogge}). From this representation as well as the Strichartz estimates \eqref{eq:mstrichartz}, it follows that for any $m \in [1, \frac{3}{2})$ and any $(p, q)$ satisfying \eqref{eq:mwaveadmissible} we have the estimate
\begin{equation}\label{eq:lineareq}
\norm{(u_{l, t_0}, \partial_t u_{l, t_0})}_{L^\infty(I, \dot{H}^m \times \dot{H}^{m-1})} + \norm{u_{l, t_0}}_{L^{2(p-1)}(\mathbb{R}^3 \times I)} \lesssim \norm{(u, \partial_t u)(t_0)}_{\dot{H}^m \times \dot{H}^{m-1}} \, .
\end{equation}
From Duhamel's principle it follows that we can write for $t \in I$ 
\begin{equation}\label{eq:duhamel}
u(t)= u_{l, t_0}(t)+ \int_{t_0}^t \frac{\sin((t-t') \sqrt{-\Delta})}{\sqrt{-\Delta}} F(t') \, \mathrm{d}t' \,.
\end{equation}
We recall from \cite[Chapter 4]{ShatahStruwe} that for $t \neq t'$ we have the explicit expression
\begin{equation}\label{eq:explicitexpr}
\frac{\sin((t-t') \sqrt{-\Delta})}{\sqrt{-\Delta}} F(t')  = \frac{1}{4\pi(t-t')} \int_{\vert x - x' \rvert= \lvert t-t' \rvert} F(t', x') \, \mathrm{d}\mathcal{H}^2( x') \, .
\end{equation}

We recall that the linear evolution enjoys asymptotic stability in the following sense.
\begin{lemma}[Asymptotic stability for the linear evolution] Let $p=5+\delta$ with $\delta \in (0,1)$. Let $u$ a solution to \eqref{eq:nlw} on $\mathbb{R}^3 \times I'$ with $\norm{u}_{L^\infty(\mathbb{R}^3 \times I')} \leq M$. Then for any $I=[t_1, t_2] \subseteq I'$ and any $t\in I' \setminus I$ we have that 
\begin{align} \label{eq:asymstab1}
\norm{u_{l, t_2}(t)-u_{l, t_1}(t)}_{L^\infty(\mathbb{R}^3)} \lesssim (EM^\frac{\delta}{2})^\frac{2p}{3(p-1)}\dist(t, I)^{-\frac{2}{p-1}}  \, .
\end{align}
\end{lemma}
\begin{proof}
From \eqref{eq:localizedenergy1} we deduce that 
$$\partial_t e(t) \geq \int_{\lvert x \rvert =t} \frac{\lvert u(y,t) \rvert^{p+1}}{p+1} \mathrm{d}\mathcal{H}^2(y)\,.$$ 
Integrating in time, by translation invariance and time reversability, we have 
\begin{equation}\label{eqn:tao}
\int_I \int_{\lvert x'-x \rvert =\lvert t'-t \rvert} \lvert u(x',t') \rvert^{p+1} \mathrm{d}\mathcal{H}^2(x') \, \mathrm{d}t' \lesssim E \,
\end{equation}
for any $(x,t) \in \mathbb{R}^3 \times I'$. Using \eqref{eq:duhamel}, we write for $t \in I' \setminus I$
\begin{align}
u_{l, t_2}(t)-u_{l, t_1}(t)
&= - \frac{1}{4\pi} \int_{t_1}^{t_2} \frac{1}{\lvert t-t' \rvert}\int_{\lvert x-x' \rvert =
 \lvert t-t' \rvert} \lvert u(x',t') \rvert^{p} \, \mathrm{d}\mathcal{H}^2(x') \, \mathrm{d}t' \, .
\end{align}
We apply H\"older with $(\frac{3(p-1)}{2p},\frac{3(p-1)}{p-3} )=(\frac{p+1+\frac{\delta}{2}}{p},\frac{p+1+\frac{\delta}{2}}{1+\frac{\delta}{2}})$ to estimate for any $x \in \mathbb{R}^3$
\begin{align}
\lvert u_{l, t_2}&(x,t)-u_{l, t_1}(x,t) \rvert \lesssim \int_{t_1}^{t_2} \frac{1}{\lvert t-t' \rvert} \int_{\lvert x-x' \rvert= \lvert t-t' \rvert} \lvert u(x',t') \rvert^{p} \, \mathrm{d}\mathcal{H}^2(x') \, \mathrm{d}t' \\
&\lesssim  \left( \int_{t_1}^{t_2} \int_{\lvert x-x' \rvert= \lvert t-t' \rvert} \lvert u \rvert^{p+1+\frac{\delta}{2}}(x', t') \, \mathrm{d}\mathcal{H}^2(x') \, \mathrm{d}t' \right)^\frac{2p}{3(p-1)} \left( \int_{t_1}^{t_2} \frac{\mathrm{d}t'}{\lvert t-t' \rvert^{\frac{3(p-1)}{p-3}-2}}\right)^\frac{p-3}{3(p-1)} \\
&\lesssim \left(\norm{u}_{L^\infty(\mathbb{R}^3 \times I)}^\frac{\delta}{2} \int_{t_1}^{t_2} \int_{\lvert x - x' \rvert= \lvert t-t' \rvert} \lvert u \rvert^{p+1} (x', t') \mathrm{d}\mathcal{H}^2(x') \, \mathrm{d}t' \right)^\frac{2p}{3(p-1)} \dist(t, I)^{-\frac{2}{p-1}} \\
&\lesssim (M^\frac{\delta}{2} E)^\frac{2p}{3(p-1)} \dist(t, I)^{-\frac{2}{p-1}} \,. \qedhere
\end{align}
\end{proof}
The importance of the above asymptotic stability lies in the following corollary. 
\begin{corollary}\label{cor:subseqintervals} Let $p=5+\delta$ with $\delta\in (0,1)$ and $I=[t_-,t_+]$. Consider a solution $(u, \partial_t u) \in L^\infty(I, \dot H^1 \cap \dot H^2 \times H^1)$ to \eqref{eq:nlw} and assume that $\norm{u}_{L^\infty(\mathbb{R}^3 \times I)} \leq M$. Consider $I_1=[t_1, t_2]$ and $I_2=[t_2, t_3]$ for any $t_- \leq t_1 <t_2 <t_3 \leq t_+$. Then
\begin{equation}
\norm{u_{l,t_3}-u_{l,t_+}}_{L^{2(p-1)}(\Gamma_+(I_1)} \lesssim \frac{\lvert I_1 \rvert^\frac{1}{2(p-1)}}{\lvert I_2 \rvert^\frac{1}{2(p-1)}} (EM^\frac{\delta}{2})^\frac{p}{6(p-1)} \norm{u}_{L^\infty(I, (\dot{H}^{s_p} \times \dot{H}^{s_p-1}))}^\frac{3}{4} \, .
\end{equation}
\end{corollary}
{\begin{proof}
We observe that the pair $(\infty, \frac{3}{2}(p-1))$ is wave-$s_p$-admissible, where we recall that $s_p:=1+ \frac{\delta}{2(p-1)}$ is the critical Sobolev regularity of \eqref{eq:nlw}. We estimate by H\"older 
\begin{align}
\norm{u_{l, t_3}- u_{l, t_+}}_{L^{2(p-1)}(\Gamma_+(I_1))} 
&\lesssim \lvert I_1 \rvert^\frac{1}{2(p-1)} \norm{u_{l, t_2}- u_{l, t_3}}_{L^\infty(\mathbb{R}^3 \times I_1)}^\frac{1}{4} \norm{u_{l,t_3}-u_{l, t_+}}_{L^\infty L^{\frac{3}{2}(p-1)}(\Gamma_+(I_1))}^\frac{3}{4} \, .
\end{align}
Observe that $v:= u_{l, t_3}- u_{l, t_+}$ solves $\Box v =0$ with $v(t_3)=u(t_3)- u_{l, t_+}(t_3)$. Hence by the Strichartz estimates \eqref{eq:mstrichartz} and \eqref{eq:lineareq} we have
\begin{align}
\norm{v}_{L^\infty L^{\frac{3}{2}(p-1)}(\Gamma_+(I_1))} &\lesssim \norm{(v, \partial_t v)(t_3)}_{(\dot{H}^{s_p} \times \dot{H}^{s_p-1})(\mathbb{R}^3)} \\
&\lesssim \norm{(u, \partial_t u)(t_3)}_{\dot{H}^{s_p} \times \dot{H}^{s_p-1}}+ \norm{(u_{l, t_+}, \partial_t u_{l, t_+})(t_3)}_{\dot{H}^{s_p} \times \dot{H}^{s_p-1}} \\
&\lesssim \norm{(u, \partial_t u)(t_3)}_{\dot{H}^{s_p} \times \dot{H}^{s_p-1}}+\norm{(u, \partial_t u)(t_+)}_{\dot{H}^{s_p} \times \dot{H}^{s_p-1}}  \\
&\lesssim \norm{(u, \partial_t u)}_{L^\infty(I, (\dot{H}^{s_p} \times \dot{H}^{s_p-1}))} \, .\qedhere
\end{align}
\end{proof}
}

\section{A reverse Sobolev inequality and mass concentration}\label{sec:rev-mass}
The section is devoted to prove that, if $u$ solves \eqref{eq:nlw}, then there exists a suitable ball with controlled size which contains an amount of $L^2$ norm, quantified in terms of $\|u\|_{L^{2(p-1)}}$ and $\|u\|_{H^{s}}$.
A key ingredient in the proof is the reverse Sobolev inequality of Tao, generalized for any $s\in (0,\frac 32)$. We present the proof for completeness, since the original argument used the fact that $p$ was integer.

\begin{prop}\label{prop:reverse-sob} Let $0<s <\frac 32$ and $\frac{1}{q}:=\frac{1}{2}-\frac{s}{3}$. Let  $f \in \dot{H}^s(\mathbb{R}^3)$
. Then there exists $x\in \mathbb{R}^3$ and $0< r \leq \frac{2}{N}$ such that 
\begin{equation}\label{eqn:concentr-ts}
\left(\frac{1}{r^{2s}} \int_{B(x,r)} f^2 (y) \, \mathrm{d}y \right)^\frac{1}{2} \gtrsim \norm{P_{\geq N} f}_{L^{q}(\mathbb{R}^3)}^{\left(\frac{3}{2s}\right)^2} \norm{f}_{\dot{H}^s}^{1-\left(\frac{3}{2s}\right)^2 } .
\end{equation}
\end{prop}
\begin{proof}
By replacing $f$ with $\tilde{f}(x):= \frac{1}{\norm{f}}_{\dot{H}^s} f(x)$ we can assume w.l.o.g. that $\norm{f}_{\dot{H}^s}=1$. \\
\\
\textit{Step 1: Let $g\in \dot{H}^s$ with $\norm{g}_{\dot{H}^s}\leq 1$. Then there exists $\bar N \in 2^\mathbb{Z}$ such that
\begin{equation}\label{eq:massconclwbLq}
\norm{g}_{L^q}^{\frac{3}{2s}} \lesssim \norm{P_{\bar N} g}_{L^q} \, ,
\end{equation}
and as a consequence
\begin{equation}\label{eq:massconclwbLinfty}
\norm{g}_{L^q}^{\left(\frac{3}{2s}\right)^2}{\bar N}^{\frac 3q} \lesssim \norm{P_{\bar N} g}_{L^\infty}  \, .
\end{equation}
}
From \eqref{eq:PaleyLittlewoodineq}, Plancherel's theorem
and the hypothesis $\norm{g}_{\dot{H}^s}\leq 1$, we infer that 
\begin{equation}\label{eq:massconcHsbound}
\sum_{N \in 2^\mathbb{Z}} N^{2s} \norm{P_N g}_{L^2}^2 \lesssim 1 \, .
\end{equation}
By interpolation, \eqref{eq:massconcHsbound} and the definition of $q$ we see that  \eqref{eq:massconclwbLinfty} is  a consequence of \eqref{eq:massconclwbLq}; indeed
\begin{equation}
\norm{P_{\bar N} g}_{L^q} \leq \norm{P_{\bar N} g}_{L^2}^\frac{2}{q} \norm{P_{\bar N} g}_{L^\infty}^{1-\frac{2}{q}} = {\bar N}^{-\frac{2s}{q}} \left({\bar N}^{2s}\norm{P_{\bar N} g}_{L^2}^2 \right)^\frac{1}{q} \norm{P_{\bar N} g}_{L^\infty}^{1-\frac{2}{q}}\lesssim {\bar N}^{-\frac{2s}{q}} \norm{P_{\bar N} g}_{L^\infty}^{\frac{2s}{3}} \, .
\end{equation}
 We are left to prove  \eqref{eq:massconclwbLq}. Let us fix $M \in \mathbb{N}$ big enough such that $\frac{q}{2} \in (M-1, M] \, $. With this choice of $M$, we ensure the subadditivity of the map $x \mapsto x^\frac{q}{2M}$. We then write using the hypothesis, \eqref{eq:PaleyLittlewoodineq}, the aforementioned subbadditivity, a reordering and H\"older 
\begin{align}\label{eqn:boh}
\norm{g}_{L^q}^q &\lesssim \int \Big(\sum_{M \in 2^\mathbb{Z}} \lvert P_M g (x)\rvert^2 \Big)^\frac{q}{2} \, \mathrm{d}x = \int \prod_{i=1}^M \Big( \sum_{N_i \in 2^\mathbb{Z}} \lvert P_{N_i} g (x) \rvert^2 \Big)^\frac{q}{2M} \, \mathrm{d}x \\
&\leq \int \prod_{i=1}^M  \sum_{N_i \in 2^\mathbb{Z}} \lvert P_{N_i} g (x) \rvert^\frac{q}{M}  \, \mathrm{d}x  \lesssim \sum_{\substack{N_1 \leq \dots \leq N_M}} \int  \prod_{i=1}^M  \lvert P_{N_i} g (x) \rvert^\frac{q}{M}  \, \mathrm{d}x  \\
&\lesssim \Big( \sup_{N \in 2^\mathbb{Z}} \norm{P_N g}_{L^q}\Big)^\frac{q(M-2)}{M} \sum_{\substack{N_1 \leq \dots \leq N_M}}  \left(\int \lvert P_{N_1} g(x) \rvert^\frac{q}{2}\lvert P_{N_M} g(x) \rvert^\frac{q}{2}  \, \mathrm{d}x \right)^\frac{2}{M} \, .
\end{align}
In all sums on $N_1 \leq \dots \leq N_M$, we intend that each $N_i$ belongs to $2^\mathbb{Z}$.
We claim that the second factor is bounded by a constant. Indeed, we estimate the last integral for fixed $N_1$ and $N_M$ using H\"older by 
\begin{align}
\left(\int \lvert P_{N_1} g(x) \rvert^\frac{q}{2}\lvert P_{N_M} g(x) \rvert^\frac{q}{2}  \, \mathrm{d}x \right)^\frac{2}{M} &\leq \left( \norm{P_{N_1} g }_{L^\infty}^\frac{M}{2} \int \lvert P_{N_1} g(x) \rvert^{\frac{q-M}{2}} \lvert P_{N_M} g(x) \rvert^\frac{q-M}{2} \lvert P_{N_M} g(x) \rvert^\frac{M}{2} \, \mathrm{d}x \right)^\frac{2}{M}  \\
&\leq \norm{P_{N_1} g }_{L^\infty}\norm{P_{N_1} g}_{L^q}^\frac{q-M}{M} \norm{P_{N_M} g}_{L^q}^\frac{q-M}{M}  \norm{ P_{N_M} g}_{L^\frac{q}{2}} \, .
\end{align}
By Bernstein's inequality \eqref{eq:Bernstein}  and the definition of $q$, we have that
\begin{equation}
\norm{P_{N_1} g }_{L^\infty} \norm{ P_{N_M} g}_{L^\frac{q}{2}}  \lesssim N_1^\frac{3}{2} N_M^{\frac{3}{2}-\frac{6}{q}} \norm{P_{N_1} g }_{L^2} \norm{ P_{N_M} g}_{L^2} = N_1^\frac{3}{2} N_M^{2s-\frac{3}{2}} \norm{P_{N_1} g }_{L^2} \norm{ P_{N_M} g}_{L^2} \, .
\end{equation}
Combining the three estimates, we deduce that 
\begin{align}\label{eq:blabla}
\norm{g}_{L^q}^q &\lesssim \Big( \sup_{N \in 2^\mathbb{Z}} \norm{P_N g}_{L^q}\Big)^{q-2}\sum_{\substack{N_1 \leq \dots \leq N_M }}  \norm{P_{N_1} g }_{L^\infty} \norm{ P_{N_M} g}_{L^\frac{q}{2}} \nonumber \\
& \lesssim
\Big( \sup_{N \in 2^\mathbb{Z}} \norm{P_N g}_{L^q}\Big)^{q-2}
 \sum_{\substack{N_1 \leq \dots \leq N_M }} N_1^{\frac{3}{2}-s} N_M^{s-\frac{3}{2}} \left( N_1^{2s}  \norm{P_{N_1} g }_{L^2}^2 + N_M^{2s} \norm{ P_{N_M} g}_{L^2}^2 \right) \, .
\end{align}
Let us consider the first addendum on the right-hand side (the second is handled analogously):
\begin{align}
 \sum_{N_1 \leq \dots \leq N_M } N_1^{\frac{3}{2}-s} N_M^{s-\frac{3}{2}}N_1^{2s}  \norm{P_{N_1} g }_{L^2}^2 
 &\leq \sum_{n_1 \in \mathbb{Z}} 2^{2n_1s} \norm{P_{2^{n_1}} g }_{L^2}^2 \sum_{n_M=n_1}^\infty  (n_M-n_1)^{M-2} 2^{-(\frac{3}{2}-s)(n_M-n_1)} \\
&\lesssim \sum_{n_1 \in \mathbb{Z}} 2^{2n_1s} \norm{P_{2^{n_1}} g }_{L^2}^2 \lesssim 1 \, , 
\end{align}
where we used that for fixed $s\in (0, \frac 32)$ the series $\norm{P_{2^{n_1}} g }_{L^2}^2 \sum_{n=0}^\infty  n^{M-2} 2^{-(\frac{3}{2}-s)n} $ converges for every $M \in \mathbb{N}\, $ as well as \eqref{eq:massconcHsbound}. We conclude from \eqref{eq:blabla} that
$\norm{g}_{L^q}^\frac{3}{2s} = \norm{g}_{L^q}^\frac{q}{q-2} \lesssim \sup_{N \in 2^\mathbb{Z}} \norm{P_N}_{L^q} \, ,$
which implies \eqref{eq:massconclwbLq}.
\\
\\
\textit{Step 2: Let $\bar N, N \in 2^\mathbb{Z}$ and define $\psi_{\bar N}:= \bar N^3 \psi( \bar N x)$ where $\psi$ is a bump function supported in $B_1(0)$ whose Fourier transform has magnitude $\sim 1$ on $B_{100}(0)$. Then we can rewrite  
\begin{equation}
P_{\bar N} P_{\geq N} f = \tilde{P}_{\bar N}( f \ast \psi_{\bar N}) \, ,
\end{equation}
where $\tilde{P}_{\bar N}$ is a Fourier multiplier which is bounded on $L^\infty$.\\}
The claimed identity of Fourier multipliers follows by setting $\mathcal{F}(\tilde{P}_{\bar N})(\xi):= \Psi(\xi / \bar N)$, where 
\begin{equation}
\Psi(\xi):=(\varphi(\xi)- \varphi(2\xi))(1-\varphi( \xi \bar N/N )) \hat{ \psi}(\xi)^{-1} \,.
\end{equation}
To verify that $\tilde{P}_{\bar N}$ is bounded on $L^\infty$, for $g\in L^\infty$ we estimate by Young and a change of variables
\begin{equation}
\norm{\tilde{P}_{\bar N} g}_{L^\infty} \lesssim \norm{\mathcal{F}^{-1}(\Psi (\xi/\bar N))}_{L^1} \norm{g}_{L^\infty} = \norm{\mathcal{F}^{-1}(\Psi)}_{L^1} \norm{g}_{L^\infty} \, .
\end{equation}
Observe that $\Psi \in C_c^\infty(\mathbb{R}^3) \subseteq \mathcal{S}(\mathbb{R}^3)$, so that $\norm{\mathcal{F}^{-1}(\Psi)}_{L^1} < + \infty \, .$
\\
\\
\textit{Step 3: Conclusion of the proof.\\}
We apply Step 1 to $g= P_{\geq N} f$ to deduce that there exist $\bar N \in 2^\mathbb{Z}$ such that 
\begin{equation}
\norm{P_{\geq N} f}_{L^q}^{\left(\frac{3}{2s}\right)^2} \bar N^\frac{3}{q} \lesssim \norm{P_{\bar N} P_{\geq N} f}_{L^\infty} \, .
\end{equation}
We observe that $\bar N \geq \frac{N}{2}$ because otherwise $P_{\bar N} P_{\geq N}f =0\,.$ By Step 2, we deduce that there exists  $x\in \mathbb{R}^3$ such that 
\begin{equation}
\norm{P_{\geq N} f}_{L^q}^{\left(\frac{3}{2s}\right)^2} { \bar N}^{ \frac 3q}
\lesssim
\lvert \psi_{\bar N} \ast f(x) \rvert 
 \leq \bar N^{\frac 3 2}\Big( \int_{B(x, 1/\bar N)} f^2(y) \, \mathrm{d}y \Big)^\frac{1}{2} \norm{ \psi}_{L^2} \, .
\end{equation}
 Combining the two inequalities, we obtain 
the claimed inequality \eqref{eqn:concentr-ts} with $r:= \frac{1}{\bar N} \in (0, \frac{2}{N}]$.
\end{proof}

The proposition above will be applied with $s=s_p$;
the choice of $s \neq 1$ is in turn fundamental in the main theorem, since it allows to give an upper bound on the $r_0$ given by the mass concentration only in terms of $E, M, \|u \|_{L^\infty \dot H^{s_p}}$.
\begin{lemma}[Mass concentration]\label{lem:massconc} Let $p=5+ \delta$ for $\delta \in (0,1)$ and let $0< \eta \leq1$. Assume $\norm{u}_{L^{2(p-1)}(\mathbb{R}^3 \times I)} \geq \eta$ and $\norm{u}_{L^\infty(\mathbb{R}^3 \times I)} \leq M$. Then, for any $1 \leq s \leq s_p:= 1+\frac{\delta}{2(p-1)}$ there exists $(x,t) \in \mathbb{R}^3 \times I$ and $r>0$ such that 
\begin{equation}\label{eq:massconcest}
\frac{1}{r^{2s}} \int_{B(x,r)} u^2(y, t) \, \mathrm{d}y \gtrsim \norm{u}_{L^\infty(I, \dot{H}^{s_p}(\mathbb{R}^3))}^{-\alpha_0} (M^{\frac \delta 2} E)^{-\alpha_1}  M^{-(s_p-s)(p-1)} \eta^{\alpha_2}  \,,
\end{equation}
where $\alpha_i=\alpha_i(s) \geq0$ are defined as $\alpha_0:=(\gamma-2)\frac{s-1}{s_p-1}$, $\alpha_1:=\frac{3}{10} \gamma (3-2s) + \frac{\gamma-2}{2}\frac{s_p-s}{s_p-1}$ and $\alpha_2:=\frac{3-2s}{5}2(p-1)\gamma$ for $\gamma := \frac{9}{2s^2}\, .$ Moreover, 
\begin{equation}\label{eq:massconcestinterval}
\lvert I \rvert \gtrsim \eta^{2(p-1)} \norm{u}_{L^\infty(I, \dot{H}^{s_p}(\mathbb{R}^3))}^{-\alpha_0'} (EM^\frac{\delta}{2})^{-\alpha_1'} M^\frac{(s-1)(p-1)}{2} r^s  \, ,
\end{equation}
where $\alpha_i'(s) \geq 0$ are defined as $\alpha_0':=2(p-1)-\frac{(s-1)(p-1)(p+1)}{\delta}$ and $\alpha_1':=\frac{(s-1)(p-1)}{\delta}\, .$
\end{lemma}

\begin{proof} Fix $1 \leq s \leq s_p= 1+\frac{\delta}{2(p-1)}$ and set $\frac{1}{q}:=\frac{1}{2}-\frac{s}{3}$, the conjugate Sobolev exponent. 
 By shrinking $I$, we can always assume that $\norm{u}_{L^{2(p-1)}(\mathbb{R}^3 \times I)} = \eta$. Recalling the proof of Lemma \ref{lem:lowerboundpotential}, we have that for any $(q,r)$ wave-1-admissible 
\begin{equation}\label{eq:wave1boundedbyenergy}
\norm{u}_{L^q L^r} \lesssim E^\frac{1}{2}\, .
\end{equation}
\textit{Step 1: We find a frequency scale $N \in 2^\mathbb{Z}$ where $\norm{P_{\geq N}f}_{L^{2(p-1)}(\mathbb{R}^3 \times I)} \gtrsim \eta$.\\}
By H\"older and Bernstein \eqref{eq:Bernstein} with exponents $2(p-1)$ and  $\frac{6(p-1)}{s+3} \in [6, q^\ast]$ we estimate
\begin{equation}
\norm{P_{<N} u}_{L^{2(p-1)}}\lesssim \lvert I \rvert^\frac{1}{2(p-1)} \norm{P_{<N} u}_{L^\infty L^{2(p-1)}}
 \lesssim \lvert I \rvert^\frac{1}{2(p-1)} N^\frac{s}{2(p-1)}\norm{u}_{L^\infty L^\frac{6(p-1)}{s+3}} \,.
\end{equation}
We observe that by interpolation and the Sobolev embedding of $\dot H^{s_p} \hookrightarrow L^\frac{3(p-1)}{2}$ that
\begin{align}
\norm{u}_{L^\infty L^\frac{6(p-1)}{s+3}} &\leq \norm{u}_{L^\infty \frac{3(p-1)}{2}}^{1-\frac{(s-1)(p+1)}{2\delta}} \norm{u}_{L^\infty L^{p+1}}^\frac{(s-1)(p+1)}{2\delta} 
\lesssim \norm{u}_{L^\infty \dot H^{s_p}}^{1-\frac{(s-1)(p+1)}{2\delta}}  (E M^\frac{\delta}{2})^\frac{(s-1)}{2\delta} M^{-\frac{(s-1)}{4}} \, .
\end{align}
Thus if we choose the frequency scale $N\in 2^\mathbb{Z}$ such that 
\begin{equation}\label{eq: lowerboundI}
\lvert I \rvert^\frac{1}{2(p-1)} N^\frac{s}{2(p-1)} \norm{u}_{L^\infty \dot H^{s_p}}^{1-\frac{(s-1)(p+1)}{2\delta}}  (E M^\frac{\delta}{2})^\frac{(s-1)}{2\delta} M^{-\frac{(s-1)}{4}}= c\eta
\end{equation}
for a universal small constant $0<c<< 1$, we can ensure that $\norm{P_{\geq N} u}_{L^{2(p-1)}(\mathbb{R}^3 \times I)} \gtrsim \eta \, $. \\ 
\\
\textit{Step 2: We deduce a lower bound of $\norm{P_{\geq N} u}_{L^\infty(I, L^q(\mathbb{R}^3))}$ in terms of $\eta, E, M$. \\}
Observe that the pair $(3, 18)$ is wave-1-admissible and that $(3, 18)$ and $(\infty, q)$ interpolate to $(\frac{5}{6}q+3, \frac56 q+3)$. Using \eqref{eq:wave1boundedbyenergy} and \eqref{eq: lowerboundI}, we have by H\"older 
\begin{align}
\eta^{2(p-1)} &\lesssim \norm{P_{\geq N} u}_{L^{2(p-1)}_{t,x}}^{2(p-1)}  \lesssim \norm{P_{\geq N}u}_{L^{\infty}_{t,x}}^{2(p-1)-(\frac56 q+3)} \norm{P_{\geq N}u}_{L^{\frac56 q+3}_{t,x}}^{\frac56 q+3} \\
&\lesssim M^{5+ 2 \delta-\frac56 q}   \norm{P_{\geq N}u}_{L^3 L^{18}}^3 \norm{P_{\geq N}u}_{L^\infty L^q}^{\frac{5}{6}q}\\
&\lesssim M^{\frac{5}{6}q (\frac{6}{q} +\frac{3}{2q} \delta -1)} (M^\frac{\delta}{2} E)^\frac{3}{2} \norm{P_{\geq N}u}_{L^{\infty}L^q}^{\frac{5}{6}q} \, ,
\end{align}
hence after some easy algebraic manipulations
\begin{align}
\norm{P_{\geq N} u}_{L^\infty L^q} &\gtrsim \eta^{\frac{12}{5q}(p-1)} (M^{\frac \delta 2} E)^{-\frac{9}{5q}} M^{-(\frac{6}{q} +\frac{3}{2q} \delta -1)} \\
&= \eta^{\frac{(3-2s)}{5}2(p-1)} (M^\frac{\delta}{2} E)^{-\frac{3}{10}(3-2s)} M^{-\frac12(s_p-s)(p-1)}
   \, .
\end{align}
\\
\textit{Step 3:} We apply the reverse Sobolev of Proposition~\ref{prop:reverse-sob} to conclude that there exists $(x,t) \in \mathbb{R}^3 \times I$ and $0< r \leq \frac 2 N$ such that
\begin{equation}\label{eq:massconcHs}
\frac{1}{r^{2s}} \int_{B(x,r)} u^2(y, t) \, \mathrm{d}y \gtrsim \norm{u}_{L^\infty(I, \dot{H}^s(\mathbb{R}^3))}^{2-\gamma} \left( \eta^{\frac{(3-2s)}{5}2(p-1)} (M^\frac{\delta}{2} E)^{-\frac{3}{10}(3-2s)} M^{-\frac12(s_p-s)(p-1)}\right)^\gamma \,,
\end{equation}
where $\gamma := \frac{9}{2s^2}\,.$ Moreover from \eqref{eq: lowerboundI} we get 
\begin{equation}
\lvert I \rvert = \frac{(c \eta)^{2(p-1)} M^\frac{(s-1)(p-1)}{2}}{\norm{u}_{L^\infty \dot H^{s_p}}^{2(p-1)-\frac{(s-1)(p-1)(p+1)}{\delta}}(EM^\frac{\delta}{2})^\frac{(s-1)(p-1)}{\delta} N^s} \gtrsim \eta^{2(p-1)}\frac{M^\frac{(s-1)(p-1)}{2}}{\norm{u}_{L^\infty \dot H^{s_p}}^{2(p-1)-\frac{(s-1)(p-1)(p+1)}{\delta}}(EM^\frac{\delta}{2})^\frac{(s-1)(p-1)}{\delta}} r^s \,.
\end{equation}
We now rewrite \eqref{eq:massconcHs}: By interpolation and energy conservation,
\begin{equation}
\norm{u}_{L^\infty \dot{H}^s} \leq E^\frac{(s_p-s)(p-1)}{\delta} \norm{u}_{L^\infty \dot{H}^{s_p}}^\frac{2(s-1)(p-1)}{\delta} \, .
\end{equation}
Observe that $\gamma \geq 2$ for $s \in (0, \frac 32)$. Thus we have that 
\begin{align}
\norm{u}_{L^\infty \dot{H}^s}^{2-\gamma} 
\gtrsim (M^\frac{\delta}{2} E)^\frac{(s_p-s)(p-1)(2-\gamma)}{\delta} \norm{u}_{L^\infty \dot{H}^{s_p}}^\frac{2(s-1)(p-1)(2-\gamma)}{\delta}    M^{\frac{(s_p-s)(p-1)(\gamma-2)}{2}} \,,
\end{align}
so that 
\begin{equation}
\frac{1}{r^{2s}} \int_{B(x,r)} u^2(y, t) \, \mathrm{d}y \gtrsim \norm{u}_{L^\infty \dot{H}^{s_p}}^{-(\gamma-2)\frac{s-1}{s_p-1}} (M^{\frac \delta 2} E)^{- \left[\frac{3}{10} \gamma (3-2s) + \frac{\gamma-2}{2}\frac{s_p-s}{s_p-1} \right]}  M^{-(s_p-s)(p-1)} \eta^{\frac{3-2s}{5}2(p-1)\gamma} \,.
\end{equation}
\end{proof}

\begin{remark}[Optimization of exponents on $\eta$, $\norm{u}_{L^\infty \dot H^{s_p}}$ and $EM^\frac{\delta}{2}$] Whilst the free powers of $M$ in \eqref{eq:massconcest} and \eqref{eq:massconcestinterval} are fixed by scaling, the other powers come from interpolation and can be optimized. Since we are not aiming at an optimal double exponential bound, we can take in Step 2 of the proof of Lemma \ref{lem:massconc} any Strichartz-1-pair $(q',r')$ (here: $(3,18)$) such that $(\infty, q)$ and $(q',r')$ interpolate to $(\tilde{r}, \tilde{r})$ with $\tilde{r} \leq 2(p-1)$.
Alternatively, to optimize the exponents $\alpha_1$ and $\alpha_2$, we first suppose that the endpoint $(2, \infty)$ was Strichartz-1-admissible, interpolate in Step 2 between $(2, \infty)$ and $(\infty, q)$ and conclude in Step 3 as before. We then approximate $(2, \infty)$ by wave-1-admissible pairs $(2+\epsilon, \frac{6(2+\epsilon)}{\epsilon})$. Letting $\epsilon \to 0$, we can reach in this way
$\alpha_1(s)=\frac{3-2s}{6} (\gamma+)+ \frac{\gamma-2}{2} \frac{s_p-s}{s_p-1}$ 
and $\alpha_2(s)=\frac{3-2s}{3}(p-1)+.
$

 In the very same way, the free exponents in Lemma \ref{lem:lowerboundpotential} can be optimized. Proceeding in this way, we would obtain the lower bound: 
\begin{equation}
\norm{u}_{L^\infty L^{p+1}}^{p+1} \gtrsim \eta^{2(p-1)+} (EM^\frac{\delta}{2})^{-(1+)} M^{-\frac{\delta}{2}} \, .
\end{equation}

\end{remark}

\section{Proof of Theorem \ref{thm:spacetime} and Corollary \ref{cor:spacetime}}
%

\begin{proof}[Proof of Theorem \ref{thm:spacetime}] Let $p=5+\delta$ with $\delta \in (0, 1)$, $J=[t_-, t_+]$ and consider a solution $(u, \partial_t u) \in L^\infty (J, ((\dot{H}^1\cap \dot{H}^2) \times H^1)(\mathbb{R}^3))$ to \eqref{eq:nlw} as in the statement. 
If either $E M^\frac{\delta}{2} < c_0$ or $L < c_0$, then we conclude by Lemma \ref{lem:smallenergy} that $\norm{u}_{L^{2(p-1)}(\mathbb{R}^3 \times J)} \leq 1 \, .$ For the rest of the argument, we thus may assume the lower bound
\begin{equation}\label{eq:lwbofEL}
\min \{EM^\frac{\delta}{2}, L \} \geq c_0 \, ,
\end{equation}
where $c_0>0$ is the universal constant given by Lemma \ref{lem:smallenergy}. 

Let $C>2c_0^{-2}$ be a universal constant that will be fixed at the end of the proof. The inequality imposed on $C$ guarantees that $CLEM^{\delta/2}>2$.

Moreover, we may assume w.l.o.g. that $\norm{u}_{L^{2(p-1)}(\mathbb{R}^3 \times J)} \geq 1$. We  then split $J$ into subintervals $J_1, \dots, J_l$ such that 
\begin{itemize}
\item $\norm{u}_{L^{2(p-1)}(\mathbb{R}^3 \times J_i)}=1$ for $i=1, \dots, l-1 \,,$ 
\item $\norm{u}_{L^{2(p-1)}(\mathbb{R}^3 \times J_l)}\leq1 \, .$
\end{itemize}
We call $J_i$ exceptional if 
\begin{equation}
\norm{u_{l, t_+}}_{L^{2(p-1)}(\mathbb{R}^3 \times J_i)} + \norm{u_{l, t_-}}_{L^{2(p-1)}(\mathbb{R}^3 \times J_i)} \geq B_{exc}^{-1} \,,
\end{equation}
for some $ B_{exc}\geq 1$ yet to be defined. We have by Strichartz estimates \eqref{eq:mstrichartz} that 
\begin{equation}
\norm{u_{l, t_+}}_{L^{2(p-1)}(\mathbb{R}^3 \times J)},\norm{u_{l, t_-}}_{L^{2(p-1)}(\mathbb{R}^3 \times J)} \lesssim L \, .
\end{equation}
In particular, $J$ cannot consist of too many exceptional intervals. More precisely, calling the number of exceptional intervals $N_{exc}:= \lvert \{ i \in \{1, \dots, l \}: J_i \text{ exceptional} \} \rvert $, we have the bound
\begin{equation}\label{eq:numberofexc}
N_{exc} \lesssim L B_{exc} \, .
\end{equation}
Between two exceptional intervals there can lie a chain $K=J_{i_0} \cup \dots \cup J_{i_1}$ of unexceptional intervals. However, since a chain $K$ of unexceptional intervals has to be confined between two exceptional intervals (or one of its endpoints is $t_-$ or $t_+$), the number of chains of unexceptional intervals $N_{chain}$ is comparable to $N_{exc}$, that is 
\begin{equation}
N_{chain} \lesssim N_{exc} \, .
\end{equation}
For a chain $K=J_{i_0} \cup \dots \cup J_{i_1}$ of unexceptional intervals, we define $N(K):= i_1+1-i_0$ to be the number of intervals it is made of. Summarizing, we have that 
\begin{equation}
\norm{u}_{L^{2(p-1) (\mathbb{R}^3 \times J)}}^{2(p-1)} \leq N_{exc} + N_{chain} \sup_{K} N(K)  \lesssim L B_{exc} (1+  \sup_{K} N(K) ) \, .
\end{equation}
The proof is thus concluded with the following lemma and with the choice of $B_{exc}$ in \eqref{eqn:Bexc-def} below.
\end{proof}

\begin{lemma}\label{lem:chain} There exists a universal constant $C \geq 1$ such that the following holds.

 Consider a solution $(u, \partial_t u) \in L^\infty (J, (\dot H^1 \cap \dot H^2 \times H^1)(\mathbb{R}^3))$ of \eqref{eq:nlw} with $p=5+\delta$, $\delta \in (0,1)$. Define $M:=\norm{u}_{L^\infty(\mathbb{R}^3 \times J)}, \,E:=E(u)$ and $L:=\norm{(u, \partial_t u)}_{L^\infty(J, (\dot{H}^{s_p} \times \dot{H}^{{s_p}-1})(\mathbb{R}^3))}$ on $J=[t_-, t_+]$ and set 
\begin{equation}\label{eqn:Bexc-def}
B_{exc}:=\left( CE M^\frac{\delta}{2} L \right)^{C (EM^\frac{\delta}{2} L)^{176}} \, .
\end{equation}
Assume that $B_{exc}^{\frac \delta 2} \leq 2$
and that
\begin{equation}\label{eq:hyplargeenergy}
\min \{E M^\frac{\delta}{2}, L\} \geq c_0 \, .
\end{equation}
Then for any chain of unexceptional intervals, that is for any $K=J_{i_0} \cup \dots \cup J_{i_1} \subseteq J$ with 
\begin{align}
&\norm{u}_{L^{2(p-1)}(\mathbb{R}^3 \times J_i)} =1 \label{eq:exinthyp1} \,,\\
&\norm{u_{l, t_+}}_{L^{2(p-1)}(\mathbb{R}^3 \times J_i)} + \norm{u_{l, t_-}}_{L^{2(p-1)}(\mathbb{R}^3 \times J_i)} \leq B_{exc}^{-1}  \label{eq:exinthyp2}\,
\end{align}for all $i \in \{i_0, \dots, i_1 \}$, we have the estimate
\begin{equation}
N(K) \lesssim B_{exc}  \, .
\end{equation}
\end{lemma}
\begin{proof}[Proof of Lemma \ref{lem:chain}] 
\textit{Step 0: Let $\alpha_0,\, \alpha_0',\, \alpha_1$ and $\alpha_1'$ be defined through Lemma \ref{lem:massconc} for $s=s_p$, that is with $\gamma:= 2 \left(\frac{3}{2s_p}\right)^2 \in [7/2, 9/2]$
\begin{equation}\label{eq:alphas}
\alpha_0 = \gamma-2 \in \big[ \frac 32 , \frac 52\big] \,, \alpha_1= \frac{6 \gamma}{5(p-1)} \in\big[\frac 34 ,\frac 32\big] \, , \alpha_0'=5 + \frac{3}{2} \delta \in \big[5,\frac {13}2\big] \, \text{ and } \, \alpha_1'=\frac{1}{2} \, .
\end{equation}
We prove that there exists $(t_0, x_0, r_0) \in K \times \mathbb{R}^3 \times (0, +\infty)$ such that 
\begin{enumerate}[(i)]
\item mass concentrates in $B(x_0,r_0)$ at time $t_0$, i.e.
\begin{equation}\label{eq:massconc}
\frac{1}{r_0^{2s_p}} \int_{B(x_0, r_0)} u^2(y, t_0) \, \mathrm{d}y \geq C_6 L^{-\alpha_0} (EM^\frac{\delta}{2})^{-\alpha_1} \, ,
\end{equation}
\item the length of the $J_i$ is uniformly bounded from below in terms of $r_0$, i.e. for all $i=i_0, \dots, i_1$
\begin{equation}\label{eq:lowerboundexint}
\lvert J_i \rvert \geq C_7 L^{-\alpha_0'} (EM^\frac{\delta}{2})^{-\alpha_1'} M^\frac{\delta}{4} r_0^{s_p} \,.
\end{equation}
\end{enumerate}
From (i), we immediately also deduce the lower bound on the mass concentration radius 
\begin{equation}\label{eq:lowerboundr0}
r_0 \gtrsim \left( L^{-\alpha_0} (EM^\frac{\delta}{2})^{-\alpha_1} \right)^\frac{p-1}{4} M^{-\frac{p-1}{2}} \, .
\end{equation}
}
By \eqref{eq:exinthyp1}, we can apply the mass concentration Lemma \ref{lem:massconc} with $\eta=1$ and $s=s_p$ to find that for any $i\in \{ i_0, \dots, i_1 \}$ there exists $(t_i, x_i, r_i) \in J_i \times \mathbb{R}^3 \times (0, +\infty)$ such that
\begin{align}
\frac{1}{r_i^{2 s_p}} \int_{B(x_i, r_i)} u^2(y, t_i) \, \mathrm{d}y &\geq C_6 L^{-\alpha_0} (EM^\frac{\delta}{2})^{-\alpha_1} \,,\\
\lvert J_i \rvert &\geq C_7 L^{-\alpha_0'} (EM^\frac{\delta}{2})^{-\alpha_1'} M^\frac{\delta}{4} r_i^{s_p} \,.
\end{align}
Defining the minimal mass concentration radius $r_0:= \min_{i \in \{i_0, \dots, i_1\}} r_i$ and calling the associated point in spacetime $(x_0, t_0)$ we reached (i) and (ii). 
The lower bound on the mass concentration radius \eqref{eq:lowerboundr0} is a consequence of the simple observation that the left-hand side of \eqref{eq:massconc} can be bounded from above, up to constants, by $ r_0^{3-2s_p} M^2 =  r_0^{\frac{4}{p-1}} M^2 \, .$ By time and space translation symmetry, we can assume that w.l.o.g. that $x_0=0$ and that $t_0=r_0$ such that $B(x_0, r_0)\times \{ t_0 \}
$ lies in the forward wave cone centered in $(0,0)$. In view of (ii) it is enough to prove that 
\begin{equation}\label{eq:goallenghtofK}
\lvert K \rvert \lesssim 
L^{-\alpha_0'} (EM^\frac{\delta}{2})^{-\alpha_1'} M^\frac{\delta}{4} B_{exc} r_0^{s_p} \, .
\end{equation}
Moreover, by time reversal symmetry, it is enough to estimate $K_+ := K \cap [t_0, +\infty)$, i.e. to show
\begin{equation}\label{eq:goallenghtofKplus}
\lvert K_+ \rvert \lesssim 
L^{-\alpha_0'} (EM^\frac{\delta}{2})^{-\alpha_1'} M^\frac{\delta}{4} B_{exc} r_0^{s_p} \, .
\end{equation}

\textit{Step 1: We find a cylinder $B(x_0, r_0) \times \tilde J_0 \subseteq \Gamma_+(K_+)$ in spacetime such that 
\begin{enumerate}[(i)]
\item mass still concentrates in $B(x_0, r_0)$ for any $t \in \tilde{J}_0$, i.e. for $t\in \tilde{J}_0$ it holds
\begin{equation}\label{eq:massconcprolonged}
\frac{1}{r_0^{2 s_p}} \int_{B(x_0, r_0)} u^2(y, t) \, \mathrm{d}y \geq \frac{C_6}{2} L^{-\alpha_0} (M^\frac{\delta}{2} E)^{-\alpha_1} \, ,
\end{equation}
\item $\tilde{J}_0$ has controlled length, i.e. $L^{-\frac{\alpha_0}{2}} (M^\frac{\delta}{2} E)^{-\frac{\alpha_1+1}2} M^\frac{\delta}{4}  r_0^{s_p} \lesssim \lvert \tilde{J}_0 \rvert \leq M^\frac{\delta}{4} r_0^{s_p}\, ,$
\item $\tilde{J}_0$ does not carry too much of the spacetime norm. More precisely, 
\begin{equation}\label{eq:upperboundspacetimeJ0}
\norm{u}_{L^{2(p-1)}(\mathbb{R}^3 \times \tilde{J}_0)}^{2(p-1)} \lesssim L^{\alpha_0' - \frac{\alpha_0}{2}} \, .
\end{equation}
\end{enumerate}
} 

The local mass is Lipschitz in time with Lipschitz constant at most $\norm{\partial_t u }_{L^\infty(J, L^2(\mathbb{R}^3))} \lesssim E^\frac{1}{2}$. More precisely, we have that 
\begin{equation}
\left \lvert \Big(\int_{B(x_0, r_0) } u^2(y, t) \, \mathrm{d}y \Big)^\frac{1}{2} -\Big(\int_{B(x_0, r_0) } u^2(y, t_0) \, \mathrm{d}y \Big)^\frac{1}{2}  \right \rvert \lesssim E^\frac{1}{2} \lvert t-t_0  \rvert \, .
\end{equation}
In particular, if $ E^\frac{1}{2} \lvert t-t_0 \rvert \leq c_1 L^{-\frac{\alpha_0}{2}} (M^\frac{\delta}{2} E)^{-\frac{\alpha_1}2} r_0^{s_p}$ for a universal $0< c_1 <<1$ yet to be chosen sufficiently small, then we still have the mass concentration on the bubble $B(x_0,r_0) \times \tilde{J}_0$, where $\tilde J_0 := [t_0, t_0 + c_1L^{-\frac{\alpha_0}{2}} (M^\frac{\delta}{2} E)^{-\frac{\alpha_1+1}2} M^\frac{\delta}{4}  r_0^{s_p}]$. More precisely, for any $t \in \tilde{J}_0$  \eqref{eq:massconcprolonged} holds. We observe that 
\begin{equation}\label{eq:sizeJ0}
\lvert \tilde{J}_0 \rvert = c_1 M^\frac{\delta}{4}L^{-\frac{\alpha_0}{2}}(EM^\frac{\delta}{2})^{-\frac 12 (\alpha_1+1)} r_0^{s_p} \leq c_1 c_0^{-\frac{1}{2}( \alpha_0 + \alpha_1+1)} M^\frac{\delta}{4} r_0^{s_p} \, ,
\end{equation}
such that we can choose $c_1 < c_0^\frac{5}{2}$ to ensure (ii). Finally, if $K_+ \subset \tilde{J}_0$ is a strict subset, then $\lvert K_+ \rvert \leq \lvert \tilde{J_0} \rvert$ and \eqref{eq:goallenghtofKplus} holds (for big enough constants in the definition of $B_{exc}$). Thus we can assume that $\tilde{J}_0 \subseteq K_+$ and hence $B(x_0, r_0) \times \tilde{J}_0 \subseteq \Gamma_+(K_+)\, .$ Finally, let us argue that $\tilde{J}_0$ cannot be covered by too many unexceptional intervals and thus cannot carry too much spacetime norm. Indeed, from \eqref{eq:lowerboundexint}, \eqref{eq:sizeJ0} and \eqref{eq:hyplargeenergy} we deduce that $\tilde J_0$ can be covered by at most 
$$\frac{c_1 L^{-\frac{\alpha_0}{2}} (EM^\frac{\delta}{2})^{-\frac{1}{2}(\alpha_1+1)} M^\frac{\delta}{4} r_0^{s_p}}{C_7 L^{-\alpha_0'} (EM^\frac{\delta}{2})^{-\alpha_1'} M^\frac{\delta}{4} r_0^{s_p}} \lesssim L^{\alpha_0'-\frac{\alpha_0}{2}} $$
many intervals of the family $\{J_i \}_{i=i_0}^{i_1}$
. Hence by \eqref{eq:exinthyp1} we deduce \eqref{eq:upperboundspacetimeJ0}.

 \textit{Step 2: Let 
 \begin{equation}\label{eq:tildeeta}
\tilde{\eta}:= c_2 ( L EM^\frac{\delta}{2} )^{-\frac 32} \in (0,c_0') \, ,
\end{equation}
with $c_0'$ defined through Remark \ref{rmk:spacetimedecayhyp3} (so that $\tilde \eta$ is admissible for the spacetime norm decay on large intervals).
For a suitable choice of the universal constant $c_2$, we truncate $\Gamma_+(K_+)$ into wave cones $\{\Gamma_+(\tilde J_i) \}_{i=1}^k$ such that 
\begin{enumerate}[(i)]
 \item each of them carries substential spacetime norm $\tilde{\eta}$, i.e. $\norm{u}_{L^{2(p-1)}(\Gamma_+(\tilde J_i))} = \tilde \eta$ for $i=1, \dots, k-1$ and $\norm{u}_{L^{2(p-1)}(\Gamma_+(\tilde{J}_k))} \leq \tilde{\eta}\, ,$
 \item the first interval is not too long, that is $\tilde{J}_1 \subseteq \tilde J_0\, .$
\end{enumerate}
}

For an $\tilde{\eta}$ yet to be chosen, we will truncate $\Gamma_+(K_+)$ into wave cones $\{\Gamma_+(\tilde J_i) \}_{i=1}^k$ such that $\norm{u}_{L^{2(p-1}(\Gamma_+(\tilde J_i))} = \tilde \eta$ for $i=1, \dots, k-1$ and $\norm{u}_{L^{2(p-1)}(\Gamma_+(\tilde{J}_k))} \leq \tilde{\eta}$. We come to the choice of $\tilde \eta$. Let us estimate the spacetime norm on the mass concentration cylinder from above
\begin{align}
\int_{\tilde J_0} \int_{B(x_0, r_0)} u^2(y, t) \, \mathrm{d}y \, \mathrm{d}t &\lesssim \left( \int_{\Gamma_+(\tilde J_0)} \lvert u \rvert^{2(p-1)} (y, t) \, \mathrm{d}y \,\mathrm{d}t \right)^\frac{1}{p-1} \lvert \tilde J_0  \rvert^\frac{p-2}{p-1}   r_0^{\frac{3(p-2)}{p-1}}
\end{align}
and from below, using \eqref{eq:massconcprolonged},
\begin{align}
\int_{\tilde J_0} \int_{B(x_0, r_0)} u^2(y, t) \, \mathrm{d}y \, \mathrm{d}t &\gtrsim \lvert \tilde J_0 \rvert L^{-\alpha_0} (M^\frac{\delta}{2} E)^{-\alpha_1} r_0^{2{s_p}} \, .
\end{align}
We have obtained, using the definition of $\tilde{J}_0$ from Step 1, that 
\begin{align}
\norm{u}_{L^{2(p-1)}(\Gamma_+(\tilde J_0))} &\gtrsim 
 (L^{-\alpha_0} (EM^\frac{\delta}{2})^{-\alpha_1} )^{\frac{2p-1}{4(p-1)}} (E^{-1} r_0^\frac{\delta}{p-1} )^\frac{1}{4(p-1)} \, .
\end{align}
Using \eqref{eq:upperboundspacetimeJ0}, we obtain an upper bound on $r_0$, that is 
\begin{align}\label{eq:upperboundr0}
r_0^\delta &\lesssim \left(L^{\alpha_0} (EM^\frac{\delta}{2})^{\alpha_1} \right)^{(2p-1)(p-1)} E^{p-1} \norm{u}_{L^{2(p-1)}(\Gamma_+(\tilde J_0))}^{4(p-1)^2} \nonumber \\
&\lesssim \left(L^{\alpha_0} (EM^\frac{\delta}{2})^{\alpha_1} \right)^{(2p-1)(p-1)} E^{p-1} L^{(\alpha_0'-\frac{\alpha_0}{2})2(p-1)} \nonumber \\
&= M^{-\frac{\delta(p-1)}{2}}L^{2(p-1)(\alpha_0(p-1)+\alpha_0')} (EM^\frac{\delta}{2})^{(p-1)(\alpha_1(2p-1)+1)} \, .
\end{align}
On the other hand, using the lower bound on $r_0$ given by \eqref{eq:lowerboundr0}, we can estimate furthermore, recalling \eqref{eq:hyplargeenergy} and \eqref{eq:alphas}, that
\begin{align}
\norm{u}_{L^{2(p-1)}(\Gamma_+(\tilde J_0))} &\gtrsim (L^{-\alpha_0} (EM^\frac{\delta}{2})^{-\alpha_1})^{\frac{2p-1}{4(p-1)}+\frac{\delta}{16(p-1)}} (EM^\frac{\delta}{2})^{-\frac{1}{4(p-1)}} \\
&= L^{-\frac{9 }{16} \alpha_0} (EM^\frac{\delta}{2})^{-(\frac{9}{16}\alpha_1 +\frac{1}{4(p-1)})} \\
&\gtrsim ( L EM^\frac{\delta}{2} )^{-\frac 3 2}\,.
\end{align}
Thus choosing 
$\tilde{\eta}:= c_2 ( L EM^\frac{\delta}{2} )^{-\frac 32},$ for a small universal constant $0<c_2<1$, we ensure that $\tilde{J_1} \subseteq \tilde J_0$. 
Choosing $c_2$ even smaller, namely $c_2 \leq  c_0' c_0^{3 }$, we ensure that $\tilde{\eta} \in (0, c_0')$, with $c_0'$ given by Remark \ref{rmk:spacetimedecayhyp3}.

\textit{Step 3: We prove the following dichotomy (analogous to \cite[Lemma 5.2]{Tao2}). Let $j \in \{1, \dots, k-1 \}$. Then, for some universal constants $C_8>8$ and $C_9<1$, either
\begin{align}\label{eq:dichotomycase1}
\lvert \tilde J_{j+1} \rvert \leq C_8 \tilde{\eta}^{-15} \lvert \tilde{J}_j \rvert 
\intertext{ or} \label{eq:dichotomycase2}
\lvert \tilde{J}_j \rvert \geq  C_9 \tilde{\eta}^5 M^\frac{\delta}{4} B_{exc} r_0^{s_p}  \, .
\end{align}
 }
 Consider two subsequent intervals $\tilde{J}_j=[t_{j-1}, t_j]$ and $\tilde{J}_{j+1}=[t_j, t_{j+1}]$ for some $j \in \{1, \dots, k-1 \}$. We have by the localized Strichartz estimates \eqref{eq:mstrichartzlocalized} (with $(\tilde{q}, \tilde{r})= (2, \frac{6(p-1)}{3p+1})$ and $v:=u-u_{l,t_{j+1}}$ solving $\Box v = \vert u \rvert^{p-1} u$ with initial datum $(v, \partial_t v)(t_{j+1})=(0,0)$) and H\"older that 
 \begin{align}
 \norm{u-u_{l, t_{j+1}}}_{L^{2(p-1)}(\Gamma_+(\tilde J_j))} &\lesssim \norm{\lvert u \rvert^{p-1} u }_{L^{\tilde q} L^{\tilde r}(\Gamma_+(\tilde J_j \cup \tilde J_{j+1}))}\\
 &\lesssim \norm{u}_{L^\infty L^{\frac{3(p-1)}{2}}(\Gamma_+(\tilde J_j \cup \tilde J_{j+1}))} \norm{u}_{L^{2(p-1)}(\Gamma_+(\tilde J_j \cup \tilde J_{j+1}))}^{p-1} \\
 &\lesssim \norm{u}_{L^\infty(\mathbb{R}^3 \times J)}^{\frac{\delta}{3(p-1)}} \norm{u}_{L^\infty L^{p+1}(\mathbb{R}^3 \times J)}^\frac{2(p+1)}{3(p-1)}\tilde{\eta}^{p-1} \\
 &\lesssim (EM^\frac{\delta}{2})^\frac{2}{3(p-1)} \tilde{\eta}^{p-1} \, .
 \end{align}
Using 
 \eqref{eq:hyplargeenergy} and \eqref{eq:tildeeta}, we have that
\begin{equation}
\tilde{\eta}^{p-2} (EM^\frac{\delta}{2} )^\frac{2}{3(p-1)} \leq c_2^{\frac{4}{9(p-1)}} L^{-\frac4 {9(p-1)}} \tilde \eta^{p-2-\frac{4}{9(p-1)}} \leq ( c_2  c_0^{ -1} )^{\frac4 {9(p-1)}} \leq  (c_0') ^{\frac4 {9(p-1)}} c_0^{\frac 8 {9(p-1)}} \leq c_0^\frac{8}{9(p-1)}\, ,
\end{equation}
where we recall that from the choice of $c_0$ in Lemma \ref{lem:smallenergy}, it is clear that it beats also the constant arising from Strichartz estimates. We infer $\norm{u-u_{l, t_{j+1}}}_{L^{2(p-1)}(\Gamma_+(\tilde J_j))} \leq  \tilde{\eta} \, .$
 Since   $\norm{u}_{L^{2(p-1)}(\Gamma_+(\tilde J_j))} = \tilde{\eta}$ by construction, the triangular inequality implies that $$\norm{u_{l, t_{j+1}}}_{L^{2(p-1)}(\Gamma_+(\tilde J_j))} \gtrsim \tilde \eta\,.$$ This now gives raise to a dichotomy: either $\norm{u_{l, t_{j+1}}- u_{l, t_+}}_{L^{2(p-1)}(\Gamma_+(\tilde J_j))} \gtrsim \tilde \eta$ or the scattering solution $u_{l, t_+}$ is non-negligible $\norm{u_{l, t_+}}_{L^{2(p-1)}(\Gamma_+(\tilde J_j))} \gtrsim \tilde{\eta}$.
\\
\textit{Case 1:} Assume $\norm{u_{l, t_{j+1}}- u_{l, t_+}}_{L^{2(p-1)}(\Gamma_+(\tilde J_j))} \gtrsim \tilde \eta$. Then in view of Corollary \ref{cor:subseqintervals}, we have 
\begin{align}
\lvert \tilde J_{j+1} \rvert &\lesssim  \tilde \eta^{-2(p-1)} (EM^\frac{\delta}{2})^\frac{p}{3} L^\frac{3(p-1)}{2} \lvert \tilde J_j \rvert \lesssim \tilde{\eta}^{-2(p-1)} (EM^\frac{\delta}{2} L)^{\frac {15}{2}}\lvert \tilde J_j \rvert 
 \lesssim \tilde{\eta}^{-15}\lvert \tilde J_j \rvert   \, ,
\end{align}
where in the second inequality we used \eqref{eq:hyplargeenergy} and in the last the definition  \eqref{eq:tildeeta}.
\\
\textit{Case 2:} Assume $\norm{u_{l, t_+}}_{L^{2(p-1)}(\Gamma_+(\tilde J_j))} \gtrsim \tilde \eta$. Recall that $K_+$ consists of unexceptional intervals. Hence we need at least $\tilde{\eta} B_{exc}$ many of them to cover $\tilde{J}_j$. Recalling the lower bound on the length of unexceptional intervals, the definition of $\tilde{\eta}$,  \eqref{eq:hyplargeenergy} and that $\alpha_0' > \alpha_1'$ from \eqref{eq:alphas}, we have 
\begin{align}
\lvert \tilde J_j \rvert &\geq C_7 \tilde{\eta} L^{-\alpha_0'} (EM^\frac{\delta}{2})^{-\alpha_1'} M^\frac{\delta}{4} B_{exc} r_0^{s_p} \\
&= C_7 \tilde{\eta} (EM^\frac{\delta}{2}L)^{-\alpha_0'}  (EM^\frac{\delta}{2})^{\alpha_0'- \alpha_1'}  M^\frac{\delta}{4} B_{exc} r_0^{s_p} \\
&\geq C_7 \tilde{\eta}^{1+\frac 23 \alpha_0'}c_2^{-\frac{2\alpha_0'}{3}} c_0^{\alpha_0'-\alpha_1'} M^\frac{\delta}{4} B_{exc} r_0^{s_p} \\
 &\geq C_9 \tilde{\eta}^{\frac{11}{2}} M^\frac{\delta}{4} B_{exc} r_0^{s_p} \,,
\end{align}
where in the last inequality we introduced a universal constant $C_9\leq C_7 c_2^{-\frac{2\alpha_0'}{3}} c_0^{\alpha_0'-\alpha_1'} $.
\\
\\
\textit{Step 4: We show that $$\lvert K_+ \rvert \leq C_9 \tilde{\eta}^{\frac{11}2} M^\frac{\delta}{4} B_{exc} r_0^{s_p} \, .$$ Since $0<\tilde \eta \leq 1$, this implies in particular that $\lvert K_+ \rvert \leq  C_9 M^\frac{\delta}{4} B_{exc} r_0^{s_p} $ and we achieved \eqref{eq:goallenghtofKplus}, thereby concluding the proof.}
\\ 
\\
Let us therefore assume by contradiction that $\lvert K_+ \rvert >  C_9 \tilde{\eta}^{\frac{11}2} M^\frac{\delta}{4} B_{exc} r_0^{s_p} $. We call $\tilde{J}_{j_1}$ the first interval for which $\lvert \tilde{J}_1 \cup \dots \cup \tilde J_{j_1} \rvert > C_9 \tilde{\eta}^{\frac{11}2} M^\frac{\delta}{4} B_{exc} r_0^{s_p} $. We observe that up to choosing the constant $C$ in the definition of $B_{exc}$ big enough, we may assume that
\begin{equation}\label{eq:assumpBexc}
\tilde \eta^{\frac{11}2} B_{exc} > \max\big\{\frac 2 {C_9}, 1\big\}\, .
\end{equation}
By the definition of $j_1$, we then have
 \begin{enumerate}[(i)]
 \item $j_1 \neq 1$. Indeed, by Step 1 and Step 2, $\lvert \tilde{J}_1 \rvert \leq \lvert \tilde J_0 \rvert \leq  M^\frac{\delta}{4} r_0^{s_p}\,.$
 \item For every $ j \in \{1, \dots, j_1-1 \}$ we have $\lvert \tilde J_{j+1} \rvert \leq C_8 \tilde \eta^{-15} \lvert \tilde{J}_j \rvert$. This follows from Step 3 since the second option in the dichotomy is ruled out.
 \end{enumerate}
Let us call $[T_1, T_2] := \tilde{J}_2 \cup \dots \tilde{J}_{j_1-1}$. We want to apply the spacetime norm decay result of Proposition \ref{cor:spacetimedecay} on $I= [T_1, T_2]$ with $\eta=\frac{\tilde \eta}{4}$. Recall that by choice of $\tilde{\eta}$ in Step 2, we have that $\frac{\tilde \eta}{4} \in (0, c_0')$ is admissible for the spacetime norm decay.  We need thus a lower bound on the length of $I$. By construction, Step 2 and (ii)
\begin{align}
C_9 \tilde{\eta}^{\frac{11}2} M^\frac{\delta}{4} B_{exc} r_0^{s_p} &\leq \lvert \tilde{J}_1 \rvert + \dots + \lvert \tilde J_{j_1} \rvert \leq M^\frac{\delta}{4} r_0^{s_p} + (T_2-T_1) + C_8 \tilde{\eta}^{-15} (T_2-T_1)\,,
\end{align}
so that
\begin{align}
T_2-T_1 &\geq \frac{1}{2C_8} \tilde{\eta}^{{\frac{41}2}} M^\frac{\delta}{4} B_{exc} r_0^{s_p} \, .
\end{align}
On the other hand, we have from Step 2 and the lower bound on $r_0$ \eqref{eq:lowerboundr0}
\begin{align}
T_1&\leq r_0 +M^\frac{\delta}{4} r_0^{s_p}= M^\frac{\delta}{4} r_0^{s_p} (1+r_0^{1-s_p} M^\frac{-\delta}{4}) \lesssim  M^\frac{\delta}{4} r_0^{s_p} \left(1+ (L^{\alpha_0} (EM^\frac{\delta}{2})^{\alpha_1})^\frac{2\delta}{(p-1)^2} \right)\\
&\lesssim M^\frac{\delta}{4} r_0^{s_p} \tilde{\eta}^{-\frac{2(\alpha_0+\alpha_1)\delta}{\gamma(p-1)^2}} \lesssim \tilde{\eta}^{-\frac{1}{4}} M^\frac{\delta}{4} r_0^{s_p}
\, .
\end{align}
Summarizing, we have obtained
 \begin{equation}\label{eq:estimateonlength}
\frac{T_2}{T_1} \geq \frac{T_2-T_1}{T_1} \geq C_{10} \tilde{\eta}^{21} B_{exc} \, .
 \end{equation}
We now claim that to reach a contradiction, it is enough to find $A$ and a constant $C \geq 1$ such that we can verify the following three requirements: \\
 \begin{enumerate}[(R1)]
 \item $A$ satifies the hypothesis \eqref{eq:spacetimedecayhypA} of Proposition \ref{cor:spacetimedecay}, that is  $A > (4C_2  \tilde{\eta}^{-1})^\frac{12(p-1)}{5} (EM^\frac{\delta}{2})^\frac{14}{5} \,,$
 \item The interval $I$ is sufficiently large to apply Proposition \ref{cor:spacetimedecay}, i.e. \eqref{eq:spacetimedecayhypI}  is verified. In view of \eqref{eq:estimateonlength}, we can enforce  \eqref{eq:spacetimedecayhypI-simple} if 
 $$B_{exc}= (C E M^\frac{\delta}{2} L)^{C (E M^\frac{\delta}{2} L)^{176}} \geq C_{10}^{-1} \tilde \eta^{-21} A^{3 (4C_2 \tilde \eta^{-1})^\frac{6(p-1)(p+1)}{5} (EM^\frac{\delta}{2})^\frac{9p+19}{10}\max\{c_0^\frac{p-1}{2}, (M^\frac{p-1}{2} T_2)^\frac{\delta}{2} \} }\,,$$
 \item Moreover $\sqrt{A}> 2 C_8 \tilde{\eta}^{-15} \, .$
 \end{enumerate}
Observe that (R3) ensures in particular that $A >4$. If (R1)-(R3) hold, we are in the position to conclude the proof following \cite{Roy}. The difficulty in the supercritical case instead relies in verifying the requirements (R1)-(R3). Indeed, if (R1)-(R3) hold, we infer from Proposition \ref{cor:spacetimedecay} that there exists $[t_1', At_1'] \subseteq \tilde{J}_2 \cup \dots \tilde{J}_{j_1-1}$ such that 
\begin{equation}
\norm{u}_{L^{2(p-1)} (\Gamma_+([t_1', At_1']))} \leq \frac{\tilde{\eta}}{4} \, .
\end{equation}  
In particular, $[t_1', At_1']$ is covered by at most two consecutive intervals of the family $\{J_j \}_{j=2}^{j_1-1}$. 
We claim that then there exists $j \in \{2, \dots, j_1-1\}$ such that 
\begin{equation}\label{eqn:roy}
\lvert \tilde{J}_j \rvert \geq \frac{\sqrt{A}}{2} \lvert \tilde{J}_{j-1} \rvert \,.
\end{equation}
Notice that in view of (R3), the claim contradicts (ii) such that we reached a contradiction. 
Indeed, assume first, that $[t_1' , At_1']$ is covered by one interval $\tilde{J}_j$ for some $j \in \{2, \dots, j_1-1 \}\, .$ Then, recalling that $A >4$, we have
\begin{equation}
\lvert \tilde{J}_j \rvert \geq t_1' (A-1) \geq \frac{A}{2} t_1' \geq \frac{A}{2} \lvert \tilde{J}_{j-1} \rvert \geq \frac{\sqrt{A}}{2} \lvert \tilde{J}_{j-1} \rvert \, .
\end{equation}
Assume now that $[t_1', At_1']$ is covered by two intervals $\tilde{J}_j= [a_j, b_j]$ and $\tilde{J}_{j+1}= [a_{j+1}, b_{j+1}]$ for some $j\in \{2, \dots, j_1-2\}$. We consider two cases. First, if $b_j \leq \sqrt A t_1'$, then $\lvert \tilde{J}_{j+1} \rvert \geq t_1' (A-\sqrt A)$ and $\lvert \tilde{J}_j \rvert \leq \sqrt A t_1'$ such that
\begin{equation}
\lvert \tilde{J}_{j+1} \rvert \geq (\sqrt A - 1) \lvert \tilde{J}_ j \rvert \geq \frac{\sqrt{A}}{2} \lvert \tilde{J}_j \rvert \, .
\end{equation}
Second, if $b_j > \sqrt A t_1'$, then $\lvert \tilde{J}_{j} \rvert \geq (\sqrt{A}-1) t_1'$ and $\lvert \tilde{J}_{j-1} \rvert \leq t_1'$ such that 
\begin{equation}
\lvert \tilde J_j \rvert \geq (\sqrt A-1) \lvert \tilde J_{j-1} \rvert \geq \frac{\sqrt A}{2} \lvert \tilde J_{j-1} \rvert \, .
\end{equation}
This proves \eqref{eqn:roy}.

To conclude the proof, we are left to verify the requirements (R1)-(R3) by choosing $A$ and $C$. We observe that the right-hand side of (R1) can be bounded from above using \eqref{eq:tildeeta} and \eqref{eq:hyplargeenergy} by  
\begin{equation}
(4 C_2 \tilde{\eta}^{-1})^\frac{12(p-1)}{5} (EM^\frac{\delta}{2})^\frac{14}{5} \leq C_{11} \tilde{\eta}^{-14}
 \, ,
\end{equation}
such that (R1) and (R3) are enforced if we set
\begin{equation}\label{eq:choiceofA}
A:= C_{12} \tilde{\eta}^{-30}\,
\end{equation}
 for  $C_{12}:= \max \{ 3C_8, C_{11} \}^2 \, .$ We are left to verify (R2). We observe that from \eqref{eq:assumpBexc}
\begin{equation}
T_2 = T_1 + (T_2-T_1) \lesssim \tilde{\eta}^{-1} M^\frac{\delta}{4} r_0^{s_p} + \tilde{\eta}^{\frac{11}{2}}M^\frac{\delta}{4} B_{exc} r_0^{s_p} \lesssim M^\frac{\delta}{4} B_{exc} r_0^{s_p} \, .
\end{equation} 
Combining this with the upper bound on $r_0$ in \eqref{eq:upperboundr0} and using \eqref{eq:alphas}, we obtain
\begin{align}
(M^\frac{p-1}{2} T_2)^\frac{\delta}{2} &\lesssim (M^\frac{8+3\delta}{4} B_{exc} r_0^{s_p} )^\frac{\delta}{2} \\ 
&\lesssim B_{exc}^\frac{\delta}{2} L^{ s_p (p-1)(\alpha_0(p-1)+\alpha_0')} (EM^\frac{\delta}{2})^{\frac{s_p}{2}(p-1)(\alpha_1(2p-1)+1)} \\
&\lesssim B_{exc}^\frac{\delta}{2} (EM^\frac{\delta}{2} L) ^{105} \\
&\leq C_{13} B_{exc}^\frac{\delta}{2} \tilde{\eta}^{-70} \,.
\end{align}
We now bound the right-hand side of (R2) from above using again \eqref{eq:tildeeta} and \eqref{eq:hyplargeenergy} by  
 \begin{align}
C_{10}^{-1} \tilde \eta^{-21} &\left(C_{12} \tilde{\eta}^{-30} \right)^{3(4C_2 \tilde{\eta}^{-1})^{42}(EM^\frac{\delta}{2})^\frac{9p+ 19}{10} \max \{ c_0^\frac{p-1}{2}, \, (M^\frac{\delta(p-1)}{2} T_2)^\frac{\delta}{2}\}} \\
 &\leq   C_{10}^{-1} \tilde \eta^{-21}\left(C_{12} \tilde{\eta}^{-30} \right)^{3 C_{13}(4C_2 \tilde{\eta}^{-1})^{42}(c_2 c_0^{-1} \tilde \eta^{-1})^\frac{9p+19}{15} \tilde{\eta}^{-70} B_{exc}^\frac{\delta}{2} } \\
 &\leq (C' EM^\frac{\delta}{2}L)^{ C' \tilde{\eta}^{-117} B_{exc}^\frac{\delta}{2}} \\
& \leq (C EM^\frac{\delta}{2}L)^{\frac C 2(EM \frac{\delta}{2}L)^{176} B_{exc}^\frac{\delta}{2}} 
  \,,
 \end{align}
for a big enough constant $C, C' \geq 1$. We now define $B_{exc}$ to be 
\begin{equation}\label{eq:definitionofBexc}
B_{exc}:= (C EM^\frac{\delta}{2}L)^{ C(EM^\frac{\delta}{2}L)^{176}} \, .
\end{equation}
for the same constant $C$. With this definition, (R2) is enforced since we assumed $B_{exc}^\frac{\delta}{2} \leq 2$.
\end{proof}
\begin{proof}[Proof of Corollary \ref{cor:spacetime}] Consider a solution $(u, \partial_t u) \in L^\infty(J, (\dot{H}^1 \cap \dot{H}^2 \times H^1)(\mathbb{R}^3))$ of \eqref{eq:nlw} with $p=5+\delta$ for $\delta \in [0,1)$ and with $\norm{(u, \partial_t u)}_{L^\infty(J, (\dot{H}^1 \cap \dot{H}^2 \times H^1)(\mathbb{R}^3))} \leq M_0\, .$ By interpolation, conservation of the energy and  the Sobolev embeddings $(\dot H^1 \cap \dot H^2)(\mathbb{R}^3) \hookrightarrow W^{1, 6}(\mathbb{R}^3)\hookrightarrow  L^\infty(\mathbb{R}^3)$, we observe
\begin{equation}
L:=\norm{(u, \partial_t u)}_{L^\infty(J, \dot H^{s_p} \times \dot H^{s_p-1})} \leq E ^{1-\frac{\delta}{2(p-1)}} M_0^\frac{\delta}{2(p-1)} \, ,
\end{equation}
\begin{equation}
M:=\norm{u}_{L^\infty(\mathbb{R}^3 \times J)} \leq C_S M_0 \, .
\end{equation}
By Theorem \ref{thm:spacetime}, if $\min\{EM^\frac{\delta}{2}, L \} < c_0\, ,$ then $\norm{u}_{L^{2(p-1)}(\R^3 \times J)} \leq 1\, .$ Otherwise, we may assume $\min\{EM^\frac{\delta}{2}, L \} \geq c_0$ and we fix $0 \leq \delta \leq \min \{1,  \frac{\ln 2}{\ln M_0}\}$. We estimate as above
\begin{align}
EM^\frac{\delta}{2} L \leq C_S^{\frac{\delta}{2}(1+ \frac{\delta}{2(p-1)})} c_0^{-\frac{\delta}{2(p-1)}} E^2 M_0^{\delta(1 - \frac{p+1}{4(p-1)})} \leq 2 C_S c_0^{-1} E^2 =: (C'E)^2
\end{align}
for $C':= \left(2 C_S c_0^{-1} \right)^\frac{1}{2}\,.$ Thus the Corollary follows, if we can meet the smallness requirement of Theorem \ref{thm:spacetime} which now reads, setting $\bar C := \sqrt{C}C',$
\begin{equation}
\left((\bar C E )^{2 C(C'E)^{352}} \right)^\delta \leq 2\,.
\end{equation}
The latter holds defining
\begin{equation}
\delta_0:= \min \left \{1, \frac{\ln 2}{\ln M_0} , \frac{\ln 2}{\ln (\bar C E) 2 C (\bar C E)^{352}} \right \} \,.
\end{equation}
Observe that $\delta_0$ depends on $M_0$ only, since $E=E(u_0, u_1)$ depends on the initial data only.
\end{proof}

\section{Proof of Theorem \ref{thm:main}}
 By time reversability, it is enough to consider forward-in-time solutions. Thanks to classical local-wellposedness and existence theory \cite{Sogge}, the proof of Theorem \ref{thm:main} consists in establishing an a priori bound on $\norm{(u, \partial_t u)}_{L^\infty([0, T], \dot H^1 \cap \dot H^2 \times H^1)}$ which is uniform in $T$.
\begin{lemma}[Local boundedness]\label{lem:localbd} Let  $\delta \in (0, 1)$, $p=5+\delta$ and consider a solution $(u, \partial_t u) \in L^\infty(I, \dot H^1\cap \dot H^2 \times H^1)$ to \eqref{eq:nlw} on $I=[t_0, t_1]$. Then there exists a universal constant $C_l\geq 1$ such that if 
\begin{equation}\label{eq:localbdsmall}
\norm{u}_{L^{2(p-1)}(\mathbb{R}^3 \times I)}^{p-1} < C_l^{-1}  \,,
\end{equation}
then 
\begin{equation}
\norm{(u, \partial_t u)}_{L^\infty(I, \dot H^1\cap \dot H^2 \times H^1)} \leq C_l \norm{(u, \partial_t u)(t_0)}_{H^1\cap \dot H^2 \times H^1} \, . 
\end{equation}
\end{lemma}
\begin{proof} For $t\in I$, define $Z(t):=\norm{(u, \partial_t u)(t)}_{H^1\cap \dot H^2 \times H^1} \, .$ By Strichartz estimates \eqref{eq:mstrichartz}, H\"older and the Sobolev embedding of $\dot H^1 \hookrightarrow L^6$ we have
\begin{align}
Z(t) &\lesssim Z(t_0)+ \norm{\lvert u \rvert^{p-1}u}_{L^2([t_0, t], L^{3/2})}+ \norm{ \nabla(\lvert u \rvert^{p-1}u)}_{L^2([t_0, t], L^{3/2})} \\
&\lesssim Z(t_0)+   \norm{ \lvert u\rvert^{p-1}}_{L^2(\mathbb{R}^3 \times [t_0, t])} \left( \norm{u}_{L^\infty([t_0, t], L^6)} + \norm{\nabla u}_{L^\infty([t_0, t], L^6)}  \right) \\
&\lesssim Z(t_0)+   \norm{ u}_{L^{2(p-1)}(\mathbb{R}^3 \times [t_0, t])}^{p-1} \sup_{t' \in [t_0, t]}Z(t') \, .
\end{align}
We set $Y(t):= \sup_{t' \in [t_0, t]} Z(t')$. Observe that $Y$ is non-decreasing, continuous,  $Y(t_0)=Z(t_0)$ and 
\begin{equation}\label{eq:conteq}
Y(t) \leq C\left(Z(t_0)+   \norm{ u}_{L^{2(p-1)}(\mathbb{R}^3 \times I)}^{p-1} Y(t)\right)
\end{equation}
for any $t \in I$. Setting $C_l:= 2C$, we have by monotonicity that $Y(t) \leq C_l Z(t_0)$ for all $t\in [t_0, \bar t]$ where $\bar t := \sup \{t \in [t_0, t_1]: Y(t) \leq C_l Z(t_0) \}$. We claim that if $ \norm{ u}_{L^{2(p-1)}(\mathbb{R}^3 \times I)}^{p-1} \leq C_l^{-1}$, then $\bar t=t_1$. Assume by contradiction that $\bar t <t_1$. By continuity $Y(\bar t) = C_l Z(t_0)$ and by the validity of \eqref{eq:conteq} at $\bar t$, we obtain 
\begin{equation}
C_l Z(t_0)=Y(\bar t) \leq C Z(t_0) + C \norm{ u}_{L^{2(p-1)}(\mathbb{R}^3 \times I)}^{p-1} Y(\bar t) < 2C Z(t_0)= C_l Z(t_0) \,,
\end{equation}
which is a contradiction.
\end{proof}
We achieve an a priori bound on $(u, \partial_t u)$ in $L^\infty([0,T], \dot{H}^1 \cap \dot{H}^2 \times H^1)$, uniform in $T$, by iterating Lemma \ref{lem:localbd} on a partition $\{I_n\}_{n=1}^N$ of $[0,T]$, where the smallness assumption \eqref{eq:localbdsmall} $$\norm{u}_{L^{2(p-1)}(\mathbb{R}^3 \times I_n)} < C_l^{-\frac{1}{p-1}}$$ is satisfied by construction. Corollary \ref{cor:spacetime} is crucial to control $N$, independent on $T$, in terms of a double exponential in $E$ and $\norm{(u, \partial_t u)}_{L^\infty \dot H^1 \cap H^2 \times H^1}^\delta$. The crucial observation is that in the limit as $\delta \to 0$, $N$ is a double exponential of the energy which in turn is controlled by the initial data only. This will allow to iterate the local bound obtained in Lemma \ref{lem:localbd} on bounded sets of initial data for $\delta$ small enough.
\begin{proof}[Proof of Theorem \ref{thm:main}] Fix $(u_0, u_1) \in \dot H^1 \cap \dot H^2 \times H^1$. Consider $(u, \partial_t u)$ solution to \eqref{eq:nlw} with $p=5 + \delta$ for $\delta \in (0,1)$. We introduce the set 
\begin{equation}
\mathcal{F}:= \left \{ T \in [0, +\infty):  \norm{(u, \partial_t u)}_{L^\infty([0, T], \dot H^1 \cap \dot H^2 \times H^1)} \leq M_0 \right \} \, ,
\end{equation}
for some $M_0= M_0(\norm{(u_0, u_1)}_{\dot H^1 \cap \dot H^2 \times H^1})$ yet to be chosen large enough. We claim that $\mathcal{F} = [0, +\infty)$. For $M_0 \geq \norm{(u_0, u_1)}_{\dot H^1 \cap \dot H^2 \times H^1}$,  it is clear that $0 \in \mathcal{F}$ and by continuity, that $\mathcal{F}$ is a closed set. We show openness. Let $T \in \mathcal{F}$. By continuity, there exists $\epsilon>0$ such that for all $T' \in [0, T+\epsilon)$  we have 
\begin{equation}
\norm{(u, \partial_t u)}_{L^\infty([0, T'], \dot H^1 \cap \dot H^2 \times H^1)} \leq 2M_0 \, .
\end{equation}
Fix such a $T'$ and let us show that $T' \in \mathcal{F}$. If $\delta \leq \delta_0(2M_0)$, with $\delta_0$ given through Corollary \ref{cor:spacetime}, then
\begin{equation}\label{eq:bd}
\norm{u}_{L^{2(p-1)}(\mathbb{R}^3 \times [0, T'])} \leq \max \left \{ 1, (C E(2M_0)^\frac{\delta}{2})^{C (E (2M_0)^\frac{\delta}{2})^{352}} \right \} \, .
\end{equation}
We can split $[0, T']$ into subintervals $\{ J_i \}_{i=1}^N$ such that 
\begin{itemize}
\item $\norm{u}_{L^{2(p-1)}(\mathbb{R}^3 \times J_i)} = \frac 12 C_{l}^{-\frac{1}{p-1}}$ for $i=1, \dots, N-1\, ,$
\item  $\norm{u}_{L^{2(p-1)}(\mathbb{R}^3 \times J_N)} \leq \frac 12 C_{l}^{-\frac{1}{p-1}} \, ,$
\end{itemize}
and we deduce by iterating Lemma \ref{lem:localbd} that 
\begin{equation}\label{eq:iteratedest}
\norm{(u, \partial_t u)}_{L^\infty([0, T'], \dot H^1 \cap \dot H^2 \times H^1)} \leq C_l^N \norm{(u_0, u_1)}_{H^1 \cap \dot H^2 \times H^1} \, .
\end{equation}
Moreover, from \eqref{eq:bd} we have the upper bound  
\begin{equation}\label{eq:estonN}
N \leq 2 C_l^\frac{1}{p-1}\max \left \{ 1 ,(C E(2M_0)^\frac{\delta}{2})^{C (E (2M_0)^\frac{\delta}{2})^{352}} \right \} \,.
\end{equation}
We want to show that with an appropriate choice of $M_0= M_0(\norm{(u_0, u_1)}_{H^1 \cap \dot H^2 \times H^1} )$ and of $\delta= \delta(\norm{(u_0, u_1)}_{H^1 \cap \dot H^2 \times H^1} )$, we have 
\begin{equation}\label{eq:goalN}
N \leq  (\ln C_l)^{-1} \ln(M_0/\norm{(u_0, u_1)}_{H^1 \cap \dot H^2 \times H^1}) \,,
\end{equation}
which in view of \eqref{eq:iteratedest} implies $\norm{(u, \partial_t u)}_{L^\infty([0, T'], \dot H^1 \cap \dot H^2 \times H^1)} \leq M_0$ concluding the proof. Observe that for $M_0$ fixed, we have that the right-hand side of \eqref{eq:estonN} as $\delta \to 0$ converges, more precisely
\begin{equation}\label{eq:estonNlimit}
\lim_{\delta \to 0}2 C_l^\frac{1}{p-1}\max \left \{ 1, (C E(2M_0)^\frac{\delta}{2})^{C (E (2M_0)^\frac{\delta}{2})^{352}} \right \} = 2 C_l^\frac{1}{4}\max \left \{ 1, (C E)^{C E^{352} } \right \} \, .
\end{equation} We now choose $M_0$ such that the right-hand side of \eqref{eq:goalN} exceeds \eqref{eq:estonNlimit} by a factor $2$, that is we choose $M_0(E, \norm{(u_0, u_1)}_{H^1 \cap \dot H^2 \times H^1})$ such that 
\begin{equation}
 (\ln C_l)^{-1}\ln( M_0 / \norm{(u_0, u_1)}_{H^1 \cap \dot H^2 \times H^1}) \geq 4 C_l^\frac{1}{4} \max \left \{ 1, (C E)^{C E^{352}} \right \}  \, 
\end{equation}
or, equivalently, 
\begin{equation}\label{eq:choiceofM0}
M_0 \geq  \norm{(u_0, u_1)}_{H^1 \cap \dot H^2 \times H^1} e^{4 C_l^\frac{1}{4}\ln C_l \max \left \{ 1, (C E)^{C E^{352}} \right \}} \, .
\end{equation}
Finally, by \eqref{eq:estonN} we can choose $\bar \delta_0= \bar  \delta_0(M_0)< \delta_0(2M_0)$ even smaller such that for all $\delta \in (0, \bar \delta_0)$ we have 
\begin{equation}\label{eq:upperboundN}
N \leq 4 C_l^\frac{1}{4}\max \left \{ 1, (C E)^{C E^{352}} \right \} \, .
\end{equation}
This finishes the proof that $F= [0,+ \infty)$ and in particular the solution $(u, \partial_t u)$ cannot blow-up.  Recalling the choice of $M_0$, we then obtain  \eqref{eq:estimateonSobolev}. As a byproduct of the upper bound \eqref{eq:upperboundN} on $N$, independent on the size of the interval, we also obtain that
\begin{equation}
\norm{u}_{L^{2(p-1)}(\mathbb{R}^3 \times [0, +\infty)} \leq \frac{1}{2} C_l^{-\frac{1}{p-1}}4 C_l^\frac{1}{4}\max \left \{ 1, (C E)^{C E^{352}} \right \} \leq 2 \max \left \{ 1, (C E)^{C E^{352}} \right \} \, ,
\end{equation}
where we used that $C_l \geq 1\, .$
\end{proof}


\begin{thebibliography}{10}

\bibitem{BahouriGerard}
H.~Bahouri and P.~Gerard. 
\newblock High frequency approximation of solutions to critical nonlinear wave equations.
\newblock {\em Amer. J. Math.}, 121(1):131--175, 1999.

\bibitem{BahouriShatah}
H.~Bahouri and J.~Shatah.
\newblock Decay estimates for the critical semilinear wave equation.
\newblock {\em Annales de l’I. H. P.}, 15(6):783--389, 1998.

\bibitem{BeceanuSoffer}
M.~Beceanu and A.~Soffer.
\newblock Large outgoing solutions to supercritical wave equations.
\newblock {\em Int. Math. Res. Not. IMRN}, 20:6201--6253, 2018.

\bibitem{Bourgain}
J.~Bourgain.
\newblock Global well-posedness of defocusing 3D critical NLS in the radial case.
\newblock {\em JAMS}, 12(1):145--171, 1999.

\bibitem{ColomboHaffter}
M.~Colombo and S.~Haffter.
\newblock Global regularity for the hyperdissipative Navier-Stokes equation below the critical order.
\newblock {\em ArXiv e-prints}, November 2019.

\bibitem{CotiVicol}
M.~Coti~Zelati and V.~Vicol.
\newblock On the global regularity for the supercritical {SQG} equation.
\newblock {\em Indiana Univ. Math. J.}, 65(2):535--552, 2016.

\bibitem{Grillakis}
M.~Grillakis.
\newblock Regularity and asymptotic behaviour of the wave equation with a critical non-linearity.
\newblock {\em Ann. of Math.}, 132(3):485--509, 1990.


\bibitem{KenigMerle}
C.~Kenig and F.~Merle.
\newblock Nondispersive radial solutions to energy supercritical non-linear wave equations, with applications.
\newblock {\em Amer. J. Math.}, 133(4),1029--1065, 2001.

\bibitem{KillipVisan}
R.~Killip and M.~Visan.
\newblock The defocusing energy-supercritical nonlinear wave equation in three space dimensions.
\newblock {\em Trans. Amer. Math. Soc.}, 363(7):3893--3934, 2011.


\bibitem{KriegerSchlag}
J.~Krieger and W.~Schlag.
\newblock Large global solutions for energy supercritical nonlinear wave equations on $\R^{3+1}\,.$
\newblock {\em J. Anal. Math.}, 133(1):91--131, 2017.

\bibitem{MiaoPeiYu}
S.~Miao, L.~Pei and P.~Yu.
\newblock On classical global solutions of nonlinear wave equations with large data.
\newblock {\em  Int. Math. Res. Not. IMRN},19:5859--5913, 2019.


\bibitem{Roy}
T.~Roy.
\newblock Global existence of smooth solutions of a 3D log-log energy-supercritical wave equation.
\newblock {\em Anal. PDE}, 2(3):261--280, 2009.

\bibitem{ShatahStruwe1}
J.~Shatah and M.~Struwe.
\newblock Regularity Results for Nonlinear Wave Equations.
\newblock{\em Ann. of Math.}, 138(3):503--518, 1993.

\bibitem{ShatahStruwe}
J.~Shatah and M. Struwe.
\newblock Geometric wave equations.
\newblock{\em Courent Lecture Notes in Mathematics 2}, 1998.

\bibitem{Sogge}
C.~D.~Sogge.
\newblock Lectures on Non-linear Wave Equations. Second edition.
\newblock {\em  International Press}, Boston, MA, 2008.

\bibitem{Struwe2}
M.~Struwe. 
\newblock Globally regular solutions to the $u^5$ Klein-Gordon equation.
\newblock {\em Ann. Scuola Norm. Sup. Pisa Cl. Sci. (4)}, 15(3):495--513,  1988.

\bibitem{Struwe3}
M.~Struwe.
\newblock Global well-posedness of the Cauchy problem for a super-critical nonlinear wave equation in two space dimensions.
\newblock {\em Math. Ann.}, 350(3):707--719, 2011.


\bibitem{Tao}
T.~Tao.
\newblock Global regularity for a logarithmically supercritical defocusing nonlinear wave equation for spherically symmetric data.
\newblock {\em J. Hyperbolic Differ. Equ.}, 4(2):259--265, 2007.

\bibitem{Tao2}
T.~Tao.
\newblock Spacetime bounds for the energy-critical nonlinear wave equation in three spatial dimensions.
\newblock {\em Dyn. Partial Differ. Equ.},3(2):93--110, 2006.

\bibitem{Tao3}
T.~Tao.
\newblock Nonlinear Dispersive Equations: Local and Global Analysis.
\newblock {\em CBMS Regional Conference Series in Mathematics, 106. American Mathematical Society}, 2006.

\bibitem{WangYu}
J.~Wang and P.~Yu. 
\newblock A large data regime for nonlinear wave equations
\newblock {\em J. Eur. Math. Soc.}, 18(3):575--622, 2016.

\end{thebibliography}
\end{document}